\documentclass[12pt]{amsart}
\usepackage[margin=1in]{geometry}

\usepackage{aliascnt}
\usepackage{amssymb}
\usepackage{amsmath}
\usepackage{enumitem}
\usepackage{xcolor}
\definecolor{darkblue}{rgb}{0.0, 0.0, 0.8}
\usepackage[colorlinks=true, allcolors=darkblue]{hyperref}
\usepackage{todonotes}
\usepackage{tabularx}
\usepackage{mathtools}
\usepackage{aliascnt}
\usepackage{autonum}
\usepackage[colorlinks=true, allcolors=darkblue]{hyperref}
\usepackage{comment}
\newtheorem{theorem}{Theorem}[section]

\newaliascnt{lemma}{theorem}
\newtheorem{lemma}[lemma]{Lemma}
\aliascntresetthe{lemma}

\newaliascnt{proposition}{theorem}
\newtheorem{proposition}[proposition]{Proposition}
\aliascntresetthe{proposition}

\newaliascnt{corollary}{theorem}
\newtheorem{corollary}[corollary]{Corollary}
\aliascntresetthe{corollary}

\newaliascnt{conjecture}{theorem}

\aliascntresetthe{conjecture}

\newaliascnt{assumption}{theorem}

\aliascntresetthe{assumption}

\newaliascnt{definition}{theorem}
\newtheorem{definition}[definition]{Definition}
\aliascntresetthe{definition}

\newaliascnt{question}{theorem}

\aliascntresetthe{question}

\newaliascnt{remark}{theorem}
\newtheorem{remark}[remark]{Remark}
\aliascntresetthe{remark}

\newaliascnt{example}{theorem}
\newtheorem{example}[example]{Example}
\aliascntresetthe{example}

\newtheorem*{notation*}{Notation}
\newtheorem*{theorem*}{Theorem}
\newtheorem*{conjecture*}{Conjecture}

\numberwithin{figure}{section}
\numberwithin{table}{section}
\numberwithin{equation}{section}

\allowdisplaybreaks

\newcommand{\bR}{\mathbb R}

\newcommand{\cG}{\mathcal G}
\newcommand{\cH}{\mathcal H}
\newcommand{\cI}{\mathcal I}

\newcommand{\cK}{\mathcal K}

\newcommand{\cM}{\mathcal M}

\newcommand{\cP}{\mathcal P}

\newcommand{\dd}{\mathrm{d}}
\newcommand{\gc}{\gamma}

\newcommand{\Diff}{\mathrm{Diff}}

\newcommand{\Evol}{\mathrm{Evol}}
\newcommand{\Isom}{\mathrm{Isom}}
\newcommand{\Kill}{\mathrm{Kill}}
\newcommand{\id}{\mathrm{id}}
\newcommand{\evol}{\mathrm{evol}}

\renewcommand{\H}{\mathcal{H}}

\def\on{\operatorname}

\newcommand{\todomartin}[1]{\todo[inline, color=green!40]{Martin: #1}}
\newcommand{\todolevin}[1]{\todo[inline, color=blue!40]{Levin: #1}}

\usepackage[capitalise,nameinlink]{cleveref}
\Crefname{theorem}{theorem}{theorems}
\Crefname{theorem}{Theorem}{Theorems}
\Crefname{lemma}{lemma}{lemmas}
\Crefname{lemma}{Lemma}{Lemmas}
\Crefname{proposition}{proposition}{propositions}
\Crefname{proposition}{Proposition}{Propositions}
\Crefname{corollary}{corollary}{corollaries}
\Crefname{corollary}{Corollary}{Corollaries}
\Crefname{conjecture}{conjecture}{conjectures}
\Crefname{conjecture}{Conjecture}{Conjectures}
\Crefname{assumption}{assumption}{assumptions}
\Crefname{assumption}{Assumption}{Assumptions}
\Crefname{definition}{definition}{definitions}
\Crefname{definition}{Definition}{Definitions}
\Crefname{question}{question}{questions}
\Crefname{question}{Question}{Questions}
\Crefname{remark}{remark}{remarks}
\Crefname{remark}{Remark}{Remarks}
\Crefname{example}{example}{examples}
\Crefname{example}{Example}{Examples}
\begin{document}

\title[On Periodic Geodesics of Half-Lie Groups]{On Periodic Geodesics of Half-Lie Groups}
\author{Martin Bauer}
\address{Department of Mathematics, Florida State University}
\email{mbauer2@fsu.edu}

\author{Levin Maier}
\address{Faculty of Mathematics and Computer Science,
	University of Heidelberg}
\email{lmaier@mathi.uni-heidelberg.de}

\begin{abstract}
A central program in infinite-dimensional Riemannian geometry is to understand which classical finite-dimensional principles remain valid beyond finite dimensions. In this article, we study the existence of periodic geodesics on Hilbert half-Lie groups equipped with strong right-invariant Riemannian metrics. Building on  recently established completeness results for this class of infinite-dimensional manifolds, we prove that every nontrivial free homotopy class contains a periodic geodesic whenever the fundamental group is nontrivial. We further establish a Lyusternik--Fet type theorem in this setting. Assuming a Palais--Smale condition modulo right translations and the existence of a nontrivial higher homotopy group, we prove the existence of a nonconstant contractible periodic geodesic. Thus, as in the finite-dimensional setting, both first and higher homotopy information continue to force the existence of periodic geodesics in infinite dimensions. In addition, we describe a reduction principle based on compact finite-dimensional totally geodesic submanifolds and apply our results to groups of Sobolev diffeomorphisms equipped with right-invariant Sobolev metrics.
\end{abstract}

\maketitle
\setcounter{tocdepth}{1}
\tableofcontents
\section{Introduction}

\paragraph{\bf Context:} 
Infinite-dimensional differential geometry has a long history that can be traced back to Riemann's Habilitationsschrift~\cite{Riemann2016}, where the idea of manifolds beyond finite dimensions already appears. Over the last decades, the subject has developed into a rich field of its own, driven both by intrinsic mathematical questions and by a growing range of applications. Prominent examples include Arnold's geometric interpretation of the Euler equations in hydrodynamics~\cite{Arnold66}, the Benamou--Brenier formulation of optimal transport~\cite{BenamouBrenier2000} together with Otto's Riemannian metric on spaces of densities~\cite{Otto2001}, elastic metrics in shape analysis~\cite{Younes1998, BauerBruverisMichor2014Overview,SrivastavaKlassen2016}, and diffeomorphic matching methods in computational anatomy~\cite{BegMillerTrouveYounes2005}. These developments have highlighted the need for a systematic understanding of geometric structures on infinite-dimensional manifolds and of the extent to which classical finite-dimensional results continue to hold in this broader setting.

The present article is part of a broader program aimed at understanding which features of classical finite-dimensional Riemannian geometry persist in infinite dimensions. Since the pioneering work of Eells, Elworthy, Omori and others (see e.g.~\cite{Palais68,Eells1966,Eliasson1967,La99,Omori1974,EellsElworthy1970} and the references therein) significant parts of Riemannian geometry have been extended to infinite dimensions, including the development of connections, geodesic equations, curvature, and variational methods. At the same time, it became clear that infinite-dimensional geometry exhibits phenomena with no finite-dimensional counterpart. Perhaps the most famous example is the failure of the Hopf--Rinow theorem, first observed by Atkin~\cite{HopfrinowfalseAktkin} and later extended by Ekeland~\cite{HopfrinowfalseEkeland}, who constructed complete infinite-dimensional Riemannian manifolds on which minimizing geodesics need not exist between arbitrary points. A second striking phenomenon arises for weak Riemannian and weak Finsler metrics, where the induced geodesic distance may vanish identically despite the infinitesimal metric being pointwise non-degenerate. In symplectic geometry, Eliashberg and Polterovich~\cite{EliashbergPolterovich1993} exhibited this phenomenon for the \(L^p\)-analogues of Hofer's metric on the Hamiltonian diffeomorphism group. Independently, Michor and Mumford~\cite{Michor_Mumford_vanishing_geodesic_distance} observed analogous vanishing phenomena for natural weak Riemannian metrics on shape spaces, and related examples have since appeared in a variety of mapping spaces and diffeomorphism groups~\cite{BauerHarmsPreston2020,BauerBruverisHarmsMichor2012,JerrardMaor2019}.  Consequently, whenever a classical geometric or variational argument relies on compactness, completeness, or metric properties that are automatic in finite dimensions, it becomes a nontrivial question whether an analogous statement remains valid in the infinite-dimensional setting.
%A second striking phenomenon arises for weak Riemannian metrics, where the induced geodesic distance may vanish identically despite the metric being pointwise non-degenerate; this was discovered by Eliashberg and Polterovich~\cite{EliashbergPolterovich1993}  for the $L^p$-Hofer metric on the symplectomorphism group and independently by Michor and Mumford in their study of shape spaces and has since been observed in a variety of mapping spaces and diffeomorphism groups.

Recently, the first author and collaborators initiated in~\cite{Bauer_2025} a systematic study of Hilbert half-Lie groups equipped with strong right-invariant Riemannian metrics. Here a half-Lie group is a smooth manifold that is also a topological group, but only right (left, resp.) translation is required to be smooth, whereas left (right, resp.) translation is merely required to be continuous~\cite{kriegl2015exotic,marquis2018half}. 
 This class includes, in particular, groups of Sobolev diffeomorphisms, which are among the most important examples of infinite-dimensional manifolds arising in geometry, hydrodynamics, and shape analysis. Although these spaces fail to be Lie groups in the classical sense, they retain enough structure to support a rich Riemannian geometry. As such, they provide a natural testing ground for the broader question of which results from finite-dimensional global Riemannian geometry survive in infinite dimensions.  A key outcome of this previous work was the establishment of a Hopf--Rinow-type theory for strong right-invariant metrics on Hilbert half-Lie groups. More precisely, it was shown that strong right-invariant metrics on Hilbert half-Lie groups are metrically and geodesically complete and that, under a mild additional closedness assumption, any two points in the same connected component can be connected by a minimizing geodesic. Thus, despite the well-known failure of the Hopf--Rinow theorem in general infinite-dimensional Riemannian geometry, half-Lie groups equipped with strong right-invariant metrics recover many of the completeness properties familiar from finite dimensions.

The existence of periodic geodesics is another central theme of global Riemannian geometry. Originating in the classical work of Hadamard~\cite{Hadamard1898} and Poincaré~\cite{Poincare1905} on geodesics on surfaces, and in Birkhoff's variational approach~\cite{Birkhoff1917}, the subject was developed throughout the twentieth century through the work of Lyusternik--Fet~\cite{LystFetThm51}, Gromoll--Meyer~\cite{GromollMeyer1969}, and Klingenberg~\cite{Klingenberg1978} into a basic meeting point of Riemannian geometry, topology, and Hamiltonian dynamics. In the compact finite-dimensional setting, classical results reveal a deep interplay between topology and geometry: nontrivial topology forces the existence of periodic geodesics through variational methods on the loop space. Two classical mechanisms for producing periodic geodesics are particularly important. On the one hand, minimizing the energy in a nontrivial free homotopy class yields periodic geodesics representing that class~\cite{Klingenberg1978}. On the other hand, minimax constructions of Lyusternik--Fet type produce contractible periodic geodesics from higher homotopical information~\cite{LystFetThm51,GromollMeyer1969,Klingenberg1978}. Building on the aforementioned Hopf--Rinow theory, the purpose of the present article is to investigate whether these topological mechanisms continue to operate in the setting of Hilbert half-Lie groups. Indeed, the main results of this article show that, under suitable assumptions, the classical relationship between topology and the existence of periodic geodesics persists in the setting of Hilbert half-Lie groups.
\subsection*{Contributions:}
Our first main result,~\Cref{thm: periodic geodesic in all homotopy classes of loops}, establishes the existence of periodic geodesics in nontrivial free homotopy classes of Hilbert half-Lie groups. More precisely, we show that if a Hilbert half-Lie group is equipped with a strong right-invariant metric satisfying the completeness assumptions developed in our previous work and if its fundamental group is nontrivial, then every nontrivial free homotopy class contains a periodic geodesic. In particular, every primitive homotopy class contains a prime periodic geodesic. Since right translations act by isometries, it follows that every infinite-dimensional half-Lie group satisfying these assumptions carries infinitely many geometrically distinct periodic geodesics. We present two different proofs of this result, each highlighting a different aspect of the theory. The first proof is geometric in nature and proceeds via the universal covering group. A nontrivial free homotopy class determines a nontrivial deck transformation, and a minimizing geodesic connecting the identity to this deck transformation projects to a periodic geodesic in the prescribed homotopy class. The second proof is variational and is based on a direct minimization argument for the energy functional within a fixed homotopy class. While the covering-space argument makes the geometric mechanism behind the result particularly transparent, the variational proof is closer in spirit to the methods underlying the Lyusternik–-Fet theorem and serves as a bridge to our later minimax arguments. Together, the two approaches provide complementary perspectives on the relationship between topology and periodic geodesics in the setting of half-Lie groups.

Our second main result, \Cref{thm:Lyusternik_half}, is a Lyusternik–-Fet type theorem for Hilbert half-Lie groups. Whereas the first theorem produces periodic geodesics in prescribed nontrivial free homotopy classes, the Lyusternik–-Fet theorem addresses the more subtle problem of finding contractible periodic geodesics. In finite dimensions, such geodesics arise from higher homotopical information via minimax constructions on the loop space. We show that an analogous mechanism persists in the setting of half-Lie groups: if the energy functional satisfies a suitable Palais--Smale condition modulo right translations and the underlying half-Lie group possesses a nontrivial higher homotopy group, then there exists a nonconstant contractible periodic geodesic. The proof follows the classical minimax philosophy. A nontrivial higher homotopy class of the loop space gives rise to a positive minimax level for the energy functional, and a general minimax principle yields a Palais--Smale sequence at this level. Under the Palais--Smale assumption, this sequence converges to a critical point of the energy, which corresponds to a contractible periodic geodesic. Consequently, under suitable compactness assumptions, higher-order topological information again forces the existence of periodic geodesics in infinite dimensions.

In addition to these fully infinite-dimensional results, we describe in~\Cref{prop_dyn_reduction} a complementary mechanism for producing periodic geodesics on half-Lie groups. Namely, we show that whenever a half-Lie group contains a compact finite-dimensional totally geodesic submanifold, classical finite-dimensional existence theorems immediately yield periodic geodesics in the ambient infinite-dimensional space. In this way, a large body of finite-dimensional closed geodesic theory can be transferred directly to the setting of half-Lie groups. This reduction principle is particularly useful in situations where the variational assumptions required by our infinite-dimensional existence theorems are difficult to verify.

The final part of the article is devoted to applications and examples. Our primary source of examples is provided by groups of Sobolev diffeomorphisms equipped with right-invariant Sobolev metrics. We show in \Cref{cor:infinitely many periodic geodesics sobolev diffeos} that strong Sobolev metrics satisfy the assumptions of our first existence theorem and therefore admit periodic geodesics in every nontrivial free homotopy class whenever the underlying diffeomorphism group has nontrivial fundamental group. As a consequence, many classical diffeomorphism groups, including diffeomorphism groups of spheres, tori, and certain circle bundles, carry infinitely many geometrically distinct periodic geodesics. We further demonstrate in \Cref{thm:periodic_geodesics_from_isometries} that periodic geodesics can also be obtained for a class of weak right-invariant Sobolev metrics. Here the argument proceeds via finite-dimensional totally geodesic submanifolds arising from symmetries of the underlying manifold. This illustrates how classical finite-dimensional closed geodesic theory can be used to generate periodic geodesics in infinite-dimensional half-Lie groups even in situations where the stronger variational framework is unavailable. Finally, while the Palais--Smale condition appearing in our Lyusternik–-Fet theorem is presently unknown for Sobolev diffeomorphism groups, we construct in~\Cref{cor:synthetic-Lyusternik} a genuinely infinite-dimensional example for which all assumptions can be verified. This shows that the theorem is not merely formal and provides a concrete setting in which higher homotopy groups force the existence of a contractible periodic geodesics.
\subsection*{Open Questions and Future Work:}
The most immediate open problem arising from the present work concerns the Palais--Smale condition appearing in our Lyusternik–-Fet theorem. While we provide a genuinely infinite-dimensional example for which this assumption can be verified, its validity for groups of Sobolev diffeomorphisms equipped with strong right-invariant Sobolev metrics remains unknown. Establishing such a compactness result would yield the existence of contractible periodic geodesics on a large and geometrically important class of infinite-dimensional manifolds.

A second natural direction concerns the extension of our results beyond the group setting. One of the key features of Hilbert half-Lie groups is that the group structure provides a rigid link between free homotopy classes and displacement in the universal cover. As a consequence, the minimization level associated with a nontrivial free homotopy class is automatically positive, which is the crucial ingredient underlying our existence theorem for periodic geodesics. For general strong Riemannian Hilbert manifolds, no such mechanism is available, and it remains unclear under which geometric or topological assumptions a nontrivial free homotopy class must carry positive energy. Understanding this question would be an important step toward a general existence theory for periodic geodesics on infinite-dimensional Riemannian manifolds.

More broadly, the results of the present article suggest that several finite-dimensional principles relating topology to global Riemannian geometry admit meaningful infinite-dimensional counterparts. A natural continuation would be to ask whether further finite-dimensional existence results for periodic geodesics, beyond the two mechanisms considered here, persist in this setting. This includes, for instance, multiplicity results, curvature- or index-theoretic criteria, and existence theorems based on more refined properties of the loop space. Determining which of these statements continue to hold in infinite dimensions, and identifying the compactness or geometric structures responsible for their validity, remains an intriguing problem for future research.

Finally, before leaving the purely Riemannian setting, one may ask whether the methods developed here extend to suitable strong Finsler metrics in infinite dimensions. From the variational point of view this is a natural intermediate step: the quadratic energy functional would be replaced by a fiberwise convex and coercive length or action functional, while the central issues would remain the positivity of the relevant action levels and an appropriate Palais--Smale compactness condition. In particular, reversible or Tonelli-type Finsler structures on Hilbert half-Lie groups would provide a natural testing ground for such an extension.

Even more generally, the broader aim is to understand periodic orbits of infinite-dimensional Hamiltonian systems. This would place the present results in closer analogy with finite-dimensional Hamiltonian dynamics, where the search for periodic orbits has been a central theme at the intersection of dynamics~\cite{AbrahamMarsden1978,MardsenRatiuMechanics,Arnold1989} and symplectic topology~\cite{HoferZehnderBook,McDuffSalamon2017}, and where min--max methods~\cite{Rabinowitz1986,Struwe2008}, KAM theory~\cite{Arnold1989}, Aubry--Mather theory~\cite{Sorrentino2015}, and Floer-theoretic techniques~\cite{McDuffSalamon2012,AudinDamian2014} have played a decisive role. In the present framework, a natural first step would be to investigate periodic orbits of Tonelli Hamiltonians, or of right-invariant magnetic systems at high energy, on Hilbert half-Lie groups. This can be viewed as a natural next step in the infinite-dimensional Aubry--Mather theory initiated in \cite{TonelliHalfLiegroups, HopfRinowHalfLiegroups}. Moreover, in analogy with \cite{Bauer_2025}, these systems satisfy a full Hopf--Rinow type theorem. Their study would therefore provide a natural testing ground for combining the compactness and variational methods developed there with the topological arguments introduced in the present article. An ultimate goal of this line of research would be an analogue of the celebrated almost-existence theorem of Hofer and Zehnder~\cite{HoferZehnder1987} and Struwe~\cite{Str90} in the present infinite-dimensional setting.
\subsection*{Acknowledgements}
M. Bauer was partially supported by NSF grant DMS-2526630 and by the Binational Science Foundation (BSF). L. Maier acknowledges funding from the Deutsche Forschungsgemeinschaft (DFG, German Research Foundation) – 281869850 (RTG 2229), 390900948 (EXC-2181/1), and 281071066 (TRR 191).

The authors made use of ChatGPT during the preparation of this article for proofreading, language refinement, and mathematical exploration. In particular, discussions with ChatGPT helped identify that a full Palais--Smale condition is incompatible with right-invariant metrics due to the action of translations, which led to the formulation of the weaker Palais--Smale condition modulo right translations used in this work. ChatGPT also  assisted in the development of the synthetic example presented in the final section. The authors take full responsibility for all mathematical statements, proofs, and conclusions contained in this article.

\section{Preliminaries}\label{s:preliminaries}
In this section, we present the basic notation and definitions of infinite-dimensional Riemannian geometry, which will be used throughout the rest of the article. For a more comprehensive introduction to the topic, we refer the reader to~\cite{La99,schmeding2022introduction,maor2022riemannian}. In particular, we will review a general minimax principle and a recent result concerning metric completeness and the existence of minimizing geodesics in an infinite-dimensional setting~\cite{Bauer_2025}.  
\subsection{Riemannian Metrics on Infinite-Dimensional Manifolds}
 We start by recalling the notion of a Riemannian metric 
in the infinite-dimensional setting. In contrast to the finite-dimensional setting we will distinguish between two different types of Riemannian metrics, which are referred to as weak and strong Riemannian metrics:
\begin{definition}[Weak and Strong Riemannian Metrics]
Let $\mathcal M$ be a smooth, infinite-dimensional manifold. 
A Riemannian metric $G$ on $\mathcal M$ is a smooth map $$G \colon T\mathcal M \times_{\mathcal M} T\mathcal M \to \bR$$ such that $G_x$ is an inner product on $ T_x\mathcal M$ for every $x \in \mathcal M$. The metric $G$ is called a \emph{strong Riemannian metric} if the inner product $G_x$ induces the manifold topology on $T_x\mathcal M$ for every $x \in M$, whereas it is called a \emph{weak Riemannian metric} if there exists $x\in \mathcal M$, such that the induced topology on $T_x\mathcal M$ is weaker than the natural manifold  topology.
\end{definition}
As the existence of a strong Riemannian metric necessarily implies that the underlying manifold is of Hilbert type, cf.~\cite{Bauer_2025}, and as we are mostly interested in strong Riemannian manifolds in this article, we can restrict our attention from here on to Hilbert manifolds. 

\subsection{Right--Invariant Metrics on Half--Lie Groups.} In this article we will focus on a particular class of infinite-dimensional manifolds, called half-Lie groups, which we introduce next:
\begin{definition}
A right (left) half-Lie group is a smooth manifold, possibly infinite-dimensional, whose underlying topological space is a topological
group, such that right (left) translations are smooth. 
\end{definition}
In analogy with the previous section, we will restrict our attention to those half-Lie groups that admit a Hilbert manifold structure. Furthermore, without loss of generality, we restrict our attention to right half-Lie groups, as the theory for left half-Lie groups is analogous, cf.~\cite{marquis2018half,Bauer_2025}. Next, we introduce the notion of $L^2$-regularity, which will be a crucial assumption for the main results of the present article:
\begin{definition}[Regular half-Lie Groups]\label{def:regular half Lie group}
Let $\mathcal G$ be a Hilbert half-Lie group.
We say that $\mathcal G$ is \emph{$L^2$–regular} if for every $X \in L^2_{\mathrm{loc}}(\bR, T_e \mathcal G)$ there exists a unique solution $ g \in W^{1,1}_{\mathrm{loc}}(\bR, \mathcal G)$ of the differential equation
\[
    \partial_t g(t) \;=\; T_e \mu_{g(t)}\, X(t), 
    \qquad g(0) = e.
\]
This solution is called the \emph{evolution} of $X$ and is denoted by $\mathrm{Evol}(X)$. 
Its evaluation at $t=1$ is denoted by $\mathrm{evol}(X)$.
\end{definition}
We note that this notion of regularity is rather mild and that, to the best of our knowledge, all natural examples of half-Lie groups satisfy this condition; this includes, in particular, groups of diffeomorphisms of finite Sobolev regularity.
Using this notion, we have the following result on completeness properties of strong right-invariant metrics on half-Lie groups:
\begin{theorem}[Completeness for half-Lie groups {\cite[Theorem 7.7]{Bauer_2025}}]\label{thm: completeness result half lie groupd}
    Let \( (\mathcal{G},G) \) be a Hilbert half-Lie group equipped with a strong right-invariant metric $G$.   Then $(\mathcal{G},G)$ is metrically and geodesically complete.
        Assume in addition that \( \mathcal{G} \) is \( L^2 \)-regular, and that for all \( x \in G \) the set
    \[
        A_{x} := \Bigl\{ \xi \in L^2([0,1], T_e G) \,\Big|\, \operatorname{evol}(\xi) = x \Bigr\}
    \]
    is weakly closed. Then there exists a $G$-minimizing geodesic between any two points in the same connected component of $(\mathcal{G},G)$.
\end{theorem}
Recently, a similar completeness result has been obtained for Hilbert manifolds, which are open subsets of a Hilbert space, assuming a certain control of the Riemannian metric in terms of the Hilbert norm, cf.~\cite[Section~2]{bauer2025completeness}.
Furthermore, we note that, in the infinite-dimensional setting considered in the present article, the Hopf–Rinow theorem famously fails~\cite{Atkin97,HopfrinowfalseEkeland} and thus the existence of minimizing geodesics is not an automatic consequence of metric completeness but had to be proven separately.  Indeed, the assumption on the weak closedness of the set $A_x$ is not needed for metric and geodesic completeness, but only for the proof of the existence of minimizers statement, which is based on the direct method of calculus of variations. 
\subsection{The Loop Space of a Strong Riemannian Manifold}
The main results of this article establish the existence of periodic
geodesics in the settings introduced in the preceding two subsections. In
proving these results, we will need to consider spaces of loops with
values in an infinite-dimensional Riemannian manifold \(\cM\).  We write
\[
    \Lambda\cM := H^1(S^1,\cM)
\]
for the free loop space of \(\cM\), where $H^1(S^1,\cM)$ denotes the space of Sobolev functions on $S^1$ with values in $\cM$. For a free homotopy class
\(\alpha\in\pi_1(\cM)\), we denote by
\[
    \Lambda_\alpha\cM
    :=
    \left\{
    \gamma\in H^1(S^1,\cM)
    \;\middle|\;
    [\gamma]=\alpha
    \right\}
\]
the corresponding connected component of the free loop space. In particular,
for the trivial class \(\alpha=[0]\), we obtain the space of contractible
\(H^1\)-loops, which we denote by
\[
    \Lambda_0\cM := \Lambda_{[0]}\cM .
\]
If \(x_0\in\cM\) is fixed, we write
\[
    \Omega_{x_0}\cM
    :=
    \left\{
    \gamma\in H^1(S^1,\cM)
    \;\middle|\;
    \gamma(0)=\gamma(1)=x_0
    \right\}
\]
for the based loop space at \(x_0\).

In the following theorem, we show that the free loop space can be equipped
with a complete Riemannian metric.
\begin{theorem}[Metric completeness of free loop space]\label{thm:loop_space_is_complete}
Let \((\cM,G)\) be a metrically complete strong Riemannian manifold. Equip
\(\Lambda\cM\) with the natural \(H^1\)-metric \(\bar G\) induced by
\(G\), i.e.
\[
    \bar G_{\gamma}(\delta\gamma,\delta\gamma)
    =
    \int_{S^1}
    G_{\gamma(t)}(\delta\gamma(t),\delta\gamma(t))\, \mathrm{d}t +\int_{S^1}
    G_{\gamma(t)}(\nabla_{\partial_t} \delta\gamma(t),\nabla_{\partial_t} \delta\gamma(t))\,\mathrm{d}t ,
\]
for $\delta\gamma\in T_\gamma\Lambda\mathcal M$, where $\nabla_{\partial_t} \delta\gamma$ denotes the covariant derivative along $\gamma$ induced by the Levi-Civita connection of $G$. 
Then \((\Lambda\cM,\bar G)\) is a metrically complete strong
Riemannian manifold. Consequently, each connected component
\(\Lambda_\alpha\cM\), \(\alpha\in\pi_1(\cM)\), is complete with respect to
the induced \(H^1\)-metric; in particular, so is
\(\Lambda_0\cM\).
\end{theorem}
We note that this result is well known in finite dimensions; however, to the best of our knowledge, it has not yet been explicitly stated in the infinite-dimensional setting. The proof is, however, essentially identical to the finite-dimensional one; see, for example,~\cite[§2]{Abbo13Lect}.
% \subsection{Ekeland’s variational principle}
% Next, we recall Ekeland’s variational principle~\cite{EKELAND1974324}:  \begin{theorem}[Ekeland variational principle, {\cite[Thm. 1.1]{EKELAND1974324}}]\label{thm:ekeland}
% Let $(X,d)$ be a complete metric space and let 
% $F:X\to \mathbb{R}\cup\{+\infty\}$ be lower semicontinuous, bounded from below, and not identically $+\infty$.
% Then for every $\varepsilon>0$ and every $x_0\in X$ with
% \[
% F(x_0)\le \inf_{x\in X} F(x) + \varepsilon,
% \]
% Then for every $\delta>0$ there exists $x_\varepsilon\in X$ satisfying:
% \begin{enumerate}[label=(\Roman*)]
% \item $F(x_\varepsilon)\le F(x_0)$;
% \item $d(x_\varepsilon,x_0)\le \varepsilon$ (after rescaling $d$ if desired);
% \item For all $x\neq x_\varepsilon$,
% \[
% F(x_\varepsilon) < F(x) + \frac{\varepsilon}{\delta}\, d(x,x_\varepsilon).
% \]
% \end{enumerate}
% If $F$ is Fr\'echet differentiable on a Banach manifold with a $C^1$ Riemannian structure, then $$\|\nabla F(x_\varepsilon)\|\le \varepsilon.$$
% \end{theorem}

\subsection{General Minimax Principle}
Next, we recall the general minimax principle, which is a useful tool for obtaining Palais--Smale sequences. We begin by recalling the notion of a Palais--Smale sequence.
\begin{definition}\label{def:def_and_cond_PS}
Let $F\in C^1(\cM,\bR)$. A sequence $(x_n)_{n \in \mathbb{N}} \subset \cM$ is called a Palais--Smale sequence of $F$ at level $c$ (or simply a $(PS)_c$ sequence) if
\[
\lim_{n\to\infty} F(x_n) = c
\quad\text{and}\quad
\lim_{n\to\infty} \dd F(x_n) = 0.
\]
The function $F$ is said to satisfy the $(PS)_c$ condition if every $(PS)_c$ sequence admits a convergent subsequence. It is said to satisfy the $(PS)$ condition if it satisfies the $(PS)_c$ condition for every $c \in \bR$.
\end{definition}

Next, we state the general minimax principle, which guarantees the existence of Palais--Smale sequences on a complete Riemannian Hilbert manifold. For a proof, we refer to \cite{Abbo13Lect} and the references therein.

\begin{theorem}[General Minimax Principle]\label{thm:General_Minimax_Principle}
Let $F$ be a $C^{1,1}$ function on a complete Riemannian Hilbert manifold $(\cM,G)$, and let $\Gamma$ be a collection of subsets of $\cM$ which is positively invariant with respect to the negative gradient flow of $F$. If the number
\[
c = \inf_{\gamma \in \Gamma}\, \sup_{\gamma} F
\]
is finite, then $F$ admits a $(PS)_c$ sequence. In particular, if $F$ satisfies the $(PS)_c$ condition, then $c$ is a critical value of $F$.
\end{theorem}
\subsection{Groups of Sobolev Diffeomorphisms Equipped with Right-Invariant  Metrics}\label{ssec:back:diffeo}
Finally, we review a particular class of half-Lie groups: groups of Sobolev diffeomorphisms equipped with right-invariant Sobolev metrics. Interest in these geometries is motivated by their role in several application-oriented areas, such as mathematical shape analysis~\cite{younes2010shapes,micheli2013sobolev,BauerBruverisMichor2014Overview} and geometric hydrodynamics~\cite{AK98,Vi08}. In~\Cref{sec:examples}, we will use them to demonstrate the applicability of the developed theory.

We denote by \((M,g)\) a closed
finite-dimensional Riemannian manifold and by \(\Diff^s(M)\) the group of
Sobolev diffeomorphisms, which for \(s>\frac{\on{dim}(M)}{2}+1\) is defined as
\[\Diff^s(M):=\left\{\varphi\in H^s(M,M): \varphi\text{ is a $C^1$-diffeomorphism}\right\}.\]
It is well known that
\(\Diff^s(M)\) is a smooth Hilbert manifold and a topological group~\cite{EM70}. To define a right-invariant Riemannian metric on $\Diff^s(M)$, we consider a positive, elliptic, self-adjoint, pseudodifferential operator $L$ of order $2r$. Under these assumptions~$L$ defines an inner product
on
\[
T_{\mathrm{id}}\Diff^s(M)=\mathfrak X^s(M)
\]
by
\begin{equation}\label{eq:def_H_r_metric_id}
G^L_{\mathrm{id}}(u,v):=
\int_M g(u,Lv)\,\mathrm{dvol}_g
\qquad
\forall\,u,v\in \mathfrak X^s(M).
\end{equation}
Extending this inner product by right translations gives the right-invariant
Riemannian metric
\begin{equation}\label{eq:def_H_r_metric}
G^L_{\varphi}(X_{\varphi},Y_{\varphi})
:=
\int_M
g\bigl(
X_{\varphi}\circ\varphi^{-1},
L(Y_{\varphi}\circ\varphi^{-1})
\bigr)\,\mathrm{dvol}_g,
\end{equation}
called an \emph{\(H^r\)-metric} on \(\Diff^s(M)\). Equivalently, this metric is
represented by the operator
\[
L_{\varphi}
=
R_{\varphi^{-1}}^{*}\circ L\circ R_{\varphi^{-1}},
\qquad
R_{\varphi}X:=X\circ\varphi .
\]
We emphasize that since \(\varphi\in\Diff^s(M)\) is only a Sobolev diffeomorphism, the right-translated inertia operator
\(L_{\varphi}\) is no longer a pseudodifferential operator with smooth symbol
in the classical sense. Consequently, the smoothness of the Riemannian metric is not guaranteed for a general inertia operator. For $L$ a  positive, elliptic, self-adjoint, pseudodifferential operators  of order $2r$ the properties of the induced metric have been studied in detail in a series of papers~\cite{misiolek2010fredholm,escher2014right,bauer2015local,BruverisVialard2017,bauer2020well}. In the following theorem we summarize the main properties, which will be of importance in Section~\ref{sec:examples}:
\begin{theorem}[Properties of Sobolev metrics on $\Diff^s(M)$~\cite{BruverisVialard2017,bauer2020well}]\label{thm:propertiesSobmetrics}
Let $(M,g)$ be either a closed Riemannian manifold or $\mathbb R^d$, let \(s>\frac{\on{dim}(M)}{2}+1\), let $0\leq r\leq s$ and let $L$ be a positive, elliptic, self-adjoint, pseudodifferential operator  of order $2r$ acting on $\mathfrak X^s(M)$. We have
\begin{enumerate}
\item $G^L$, defined via \eqref{eq:def_H_r_metric_id}, is a smooth Riemannian metric on $\Diff^s(M)$, which is weak if $r<s$ and strong if $r=s$. 
\item For any $r>\frac{1}{2}$ the geodesic equation of $G^L$ on $\Diff^s(M)$ is locally well-posed in the sense of Hadamard. 
\item For $s=r$ the space $(\Diff^s(M), G^L)$ is geodesically and metrically complete and for any $\varphi_1,\varphi_2\in\Diff(M)$ there exists a minimizing geodesic in $\Diff^s(M)$  connecting them. 
\end{enumerate}\end{theorem}

\begin{example}[Inertia operator defined via (fractional) powers of the Laplacian]\label{ex:laplace_inertia}
As a basic example, let \(\Delta_g=-\nabla^*\nabla\) denote the (connection) Laplace operator
induced by the Riemannian metric \(g\). Then for any $0\leq r\leq s$ the operator
\[
    L=(1-\Delta_g)^r
\]
is an admissible inertia operator of order $2r$ and the induced Riemannian metric is given at the identity via:
\begin{equation}
G^L_{\mathrm{id}}(u,v)=
\int_M g(u,(1-\Delta_g)^rv)\,\mathrm{dvol}_g
=\int_M g((1-\Delta_g)^{\frac r2}u,(1-\Delta_g)^{\frac r2}v)\,\mathrm{dvol}_g.
\end{equation}
\end{example}

\section{On periodic Geodesics of Half-Lie groups}\label{sec:halfLie}
In this section we prove the main results of the present article; namely that half--Lie groups \( \cG \) with nontrivial homotopy groups equipped with strong right-invariant metrics carry periodic geodesics.  The main statements of this section are contained in \Cref{thm: periodic geodesic in all homotopy classes of loops}, \Cref{thm:Lyusternik_half} and \Cref{prop_dyn_reduction}.
%\todolevin{adapt intro of section to the result established, also may add Cref to the theorem to imporve readability.}
\subsection{Periodic Geodesics on Half-Lie Groups with Non-Vanishing First Homotopy Group.}
We start by showing the existence of periodic geodesics under the assumption that there exists a nontrivial homotopy class \( [\eta] \in \pi_1(\cG) \):
%in each nontrivial homotopy class a periodic geodesic of \( (G, \mathcal{G}) \):
%We start by fomrulating the existence result:
\begin{theorem}[Existence of periodic geodesics on half-Lie groups with non-vanishing first homotopy group]\label{thm: periodic geodesic in all homotopy classes of loops}
 Let $(\cG,G)$ be a half-Lie group equipped with a strong right-invariant metric $\mathcal{G}$. Assume in addition that 
 \begin{enumerate}[label=(\alph*)]
     \item\label{cond:L_2 regular}\(\mathcal{G}\) is \(L^2\)-regular,
     \item\label{cond:weaklyclosed} for every homotopy class \([\alpha]\in\pi_1(\cG,e)\) of based loops, the set
\[A_{e,\alpha}:=\{\xi \in L^2([0,1], T_e\mathcal{G}) \mid
\operatorname{evol}(\xi)=e,\ \operatorname{Evol}(\xi)\in[\alpha]\}\]
is weakly closed,
\item\label{cond:non_vanishing_homotopy} and that $\pi_1(\cG)\neq 0$.
 \end{enumerate}
 Then for every nontrivial homotopy class $[\eta]\in \pi_1(\cG)$, there exists a closed geodesic $\gamma$ of $(\cG,G)$ representing the class $[\eta]$. If the class $[\eta] \in \pi_1(\cG)$ is primitive, then the geodesic $\gamma$ is prime.
\end{theorem}
Before delving into the proof of
\Cref{thm: periodic geodesic in all homotopy classes of loops}, we will deduce as a corollary the existence of
infinitely many geometrically distinct periodic geodesics in  the infinite-dimensional setting:
%We close this subsection by observing that, in
%of \Cref{thm: periodic geodesic in all homotopy classes of loops}, the above
%construction yields infinitely many geometrically distinct periodic geodesics.
\begin{corollary}[Infinitely many geometrically distinct periodic geodesics]\label{cor:infinitely_many_periodic_geodeisc_in_each_homotopy_class}
Let \((\cG,G)\) be as in
\Cref{thm: periodic geodesic in all homotopy classes of loops}. If \(\cG\) is
infinite-dimensional, then there exist infinitely many geometrically distinct
periodic geodesics of \((\cG,G)\).
\end{corollary}

\begin{proof}
By \Cref{thm: periodic geodesic in all homotopy classes of loops}, there exists
a periodic geodesic \(\gamma\). Since the metric is right-invariant, every right
translate of \(\gamma\) is again a periodic geodesic.

Set \(\cK := \operatorname{im}(\gamma)\). Suppose that the collection \(\{ \cK g \mid g\in \cG_0 \}\) were finite, where $\cG_0$ denotes the connected component of the identity in $\cG$. Then the stabilizer
\[
\operatorname{Stab}_{\cG_0}(\cK)
:=
\{ g\in \cG_0 \mid \cK g = \cK \}
\]
would be a closed subgroup of finite index in the connected group \(\cG_0\),
and hence equal to \(\cG_0\). Indeed, its complement is a finite union of cosets
of \(\operatorname{Stab}_{\cG_0}(\cK)\). Since
\(\operatorname{Stab}_{\cG_0}(\cK)\) is closed, all these cosets are closed.
Thus \(\operatorname{Stab}_{\cG_0}(\cK)\) is open. Hence it is a nonempty
clopen subset of the connected space \(\cG_0\), and therefore equals
\(\cG_0\). Thus, for any \(k\in \cK\), one obtains
\[
k\cG_0 \subseteq \cK.
\]
Hence \(\cK\) would contain the connected component \(k\cG_0\) of \(\cG_0\).
This is impossible, since \(\cK\) as the image of $\gc$ is compact, whereas
no connected component of an infinite-dimensional half-Lie group is compact.
To see this, note that  each connected component of \(\cG\) is an open submanifold of \(\cG\), and thus a Hilbert
manifold modeled on the same infinite-dimensional Hilbert space. Since a compact
Hilbert manifold is necessarily finite-dimensional, this proves the above claim. As a consequence we obtain that the family
\[
\{ \operatorname{im}(\gamma)g \mid g\in \cG_0 \}
\]
is infinite. Thus the right translates of \(\gamma\) yield infinitely many
geometrically distinct periodic geodesics.
\end{proof}

In the following, we will provide two different proofs of
\Cref{thm: periodic geodesic in all homotopy classes of loops}: one is based on
the direct method of the calculus of variations, while the other is more geometric and proceeds by
lifting the problem to the universal cover of \(\cG\). 
\begin{proof}[Proof of \Cref{thm: periodic geodesic in all homotopy classes of loops} using the direct method of calculus of variations]
 We begin by associating to each nontrivial homotopy class
\([\alpha]\in \pi_1(\cG)\setminus\{0\}\) its minimal energy value
\begin{equation}\label{eq:min_energy_value_homotopy_class}
    c([\alpha])
    :=
    \inf_{\gamma\in[\alpha]} E(\gamma)
    =
    \inf_{\gamma\in[\alpha]}
    \frac{1}{2}\int_0^1
    G_{\gamma(t)}(\partial_t\gamma,\partial_t\gamma)\,\mathrm \mathrm{d} t .
\end{equation}
Since the metric $G$ is strong, it follows from \cite{La99} that the local
injectivity radius at each point is positive. By right-invariance of $G$, we
therefore conclude that the global injectivity radius
\[
\rho := \operatorname{inj}(\cG,G)
\]
of $(\cG,G)$ is positive. Hence, by standard arguments, as in the
finite-dimensional case, every noncontractible loop has length, and therefore
energy, uniformly bounded away from zero. In particular,
\[
c([\alpha])>0.
\]
Alternatively, this positivity follows from
\Cref{prop: minmx value homtopy class} presented later.

Let $(\gamma_n)\subseteq \Lambda_{[\alpha]}\cG$ be a minimizing sequence, that is,
\begin{equation}\label{eq:minimization_sequence_c_alpha}
    E(\gamma_n)\to c([\alpha]) \qquad \text{as } n\to\infty .
\end{equation}
For each \(n\), we normalize the base point of \(\gamma_n\) by setting
\begin{equation}\label{eq:bar_gamma_n}
    \bar\gamma_n(t):=\gamma_n(t)\gamma_n(0)^{-1}.
\end{equation}

This uses only the fixed right translation by the element
\(\gamma_n(0)^{-1}\). In particular, no smoothness of the normalization map on
the loop space is required. We have that $\bar\gamma_n(0)=\bar\gamma_n(1)=e$. 
Moreover, since \(G\) is right-invariant it holds that
\begin{equation}\label{eq:e_tilde_gamma_n_equal_E_gamma}
    E(\bar\gamma_n)=E(\gamma_n).
\end{equation}
Under the standard identification of free homotopy classes in a topological
group with based homotopy classes at the identity, the loop \(\bar\gamma_n\)
represents the based class \([\alpha]\). Denoting by
\[
u_n:=\partial_t\bar\gamma_n\circ\bar\gamma_n^{-1}
\]
the right logarithmic derivative of \(\bar\gamma_n\), we obtain by using the \(L^2\)-regularity of
\(\cG\) that
\[
u_n\in A_{e,[\alpha]}\subseteq L^2([0,1],T_e\cG).
\]
By right-invariance of \(G\), we have
\[
E(\bar\gamma_n)
=
\frac{1}{2}\int_0^1
G_{\bar\gamma_n(t)}(\partial_t\bar\gamma_n,\partial_t\bar\gamma_n)\,\mathrm dt
=
\frac{1}{2}\int_0^1
G_e(u_n,u_n)\,\mathrm dt .
\]
Thus by \eqref{eq:e_tilde_gamma_n_equal_E_gamma} and \eqref{eq:minimization_sequence_c_alpha} we have \(E(u_n)\to c([\alpha])\). In particular, \((u_n)\) is bounded in the
Hilbert space \(L^2([0,1],T_e\cG)\). Hence, after passing to a subsequence, we
may assume that
\[
u_n \rightharpoonup u
\qquad\text{in } L^2([0,1],T_e\cG).
\]
By Assumption~\ref{cond:weaklyclosed}, the set \(A_{e,[\alpha]}\) is weakly
closed, and therefore
\[
u\in A_{e,[\alpha]}.
\]
Moreover, by weak lower semicontinuity of the \(L^2\)-norm it holds that
\[
E(u)
\leq
\liminf_{n\to\infty} E(u_n)
=
c([\alpha]).
\]
Since \(u\in A_{e,[\alpha]}\), we also have \(c([\alpha])\leq E(u)\). Hence
\[
E(u)=c([\alpha]).
\]
Since \(u\in A_{e,[\alpha]}\), we have for \(\gamma:=\Evol(u)\) that
\(\gc(0)=\gc(1)=e\), and that \(\gamma\) represents the based homotopy
class \([\alpha]\). Thus, viewed as a free loop, we have \(\gamma\in\Lambda_{[\alpha]}\cG.\) Moreover, by right-invariance of \(G\),
\[
E(\gamma)
=
\frac{1}{2}\int_0^1 G_e(u,u)\,\mathrm dt
=
E(u)
=
c([\alpha]).
\]

We now show that \(\gamma\) is a global minimizer of \(E\) on the free homotopy
component \(\Lambda_{[\alpha]}\cG\). Let
\(\eta\in\Lambda_{[\alpha]}\cG\). As in \eqref{eq:bar_gamma_n} above, we define
\[
\bar\eta(t):=\eta(t)\eta(0)^{-1}.
\]
Hence \(\bar\eta\) is a based loop at \(e\),
represents the based class \([\alpha]\), and
for its right logarithmic derivative \(v:=\partial_t\bar\eta\circ\bar\eta^{-1}\)
it holds, by right-invariance, that
\[
v\in A_{e,[\alpha]} \qquad \text{and} \qquad E(\eta)=E(v)=E(\bar\eta).
\]
Since \(u\) realizes the infimum over \(A_{e,[\alpha]}\), we obtain
\[
E(\gamma)=E(u)\leq E(v)=E(\bar\eta)=E(\eta).
\]
Thus \(\gamma\) is a global minimizer of \(E\) on the free homotopy component
\(\Lambda_{[\alpha]}\cG\).

Since \(\Lambda_{[\alpha]}\cG\) is a connected component of the Hilbert manifold
\(\Lambda\cG\), it is open in \(\Lambda\cG\). Hence \(\gamma\) is a local
minimizer of \(E\) on the full free loop space. Since \(E\) is \(C^1\), it
follows that
\[
\mathrm dE(\gamma)=0.
\]
Thus \(\gamma\) is a critical point of the free-loop energy functional. By the
standard first variation formula, \(\gamma\) is therefore a periodic geodesic.
Since \(E(\gamma)=c([\alpha])>0\), it is nonconstant. Finally, by the classical
iteration argument, if \([\alpha]\) is primitive, then \(\gamma\) is prime. This
finishes the proof.
\end{proof}
In the following we present an alternative, more geometric proof of
\Cref{thm: periodic geodesic in all homotopy classes of loops}. We first
introduce some additional notation. Let
\[
\pi:\widetilde{\cG}\to\cG
\]
be the universal covering homomorphism and equip \(\widetilde{\cG}\) with the
pullback metric \(\widetilde G:=\pi^*G\). Then \(\widetilde G\) is again a
strong right-invariant metric, and \(\pi\) is a local isometry. Moreover,
\[
\Gamma:=\ker(\pi)\subseteq Z(\widetilde{\cG})
\]
is a discrete central subgroup, naturally isomorphic to \(\pi_1(\cG,e)\). Here $Z(\widetilde{\cG})$ denotes the centralizer of $\widetilde\cG$.

We will show that the existence of minimizing geodesics can be lifted to the
universal cover. For this, we first record the following weak closedness
consequence of condition~\ref{cond:weaklyclosed} in
\Cref{thm: periodic geodesic in all homotopy classes of loops}.

\begin{lemma}[Weak closedness for arbitrary endpoint classes]
\label{lem:weakly-closed-arbitrary-endpoints}
Let \((\cG,G)\) be as in
\Cref{thm: periodic geodesic in all homotopy classes of loops}. Then, for
every \(x\) in the connected component of the identity in \(\cG\) and every homotopy class \(\alpha\) of curves connecting
\(e\) and \(x\), the set
\[
A_{x,\alpha}
:=
\bigl\{
\xi\in L^2([0,1],T_e\cG)
\mid
\operatorname{evol}(\xi)=x,\ 
\operatorname{Evol}(\xi)\in\alpha
\bigr\}
\]
is weakly closed.
\end{lemma}
\begin{proof}
Fix \(x\in\cG\) and a homotopy class \([\alpha]\) of curves from \(e\) to \(x\).
Choose a smooth curve \(\beta\colon[0,1]\to\cG\) from \(x\) to \(e\), and
denote by \([\alpha*\beta]\) the induced based loop class at \(e\).

Let \((u_n)\subset A_{x,[\alpha]}\) be weakly convergent in
\(L^2([0,1],T_e\cG)\), say \(u_n\rightharpoonup u\). We define an affine map
\[
\Phi_\beta\colon L^2([0,1],T_e\cG)\to L^2([0,1],T_e\cG)
\]
by
\[
(\Phi_\beta v)(t)
=
\begin{cases}
2v(2t), & t\in[0,\frac12],\\[0.4em]
2\,\partial_t\beta(2t-1)\circ\beta(2t-1)^{-1},
& t\in[\frac12,1].
\end{cases}
\]
This is the operation on Eulerian velocities corresponding to first running
the curve with Eulerian velocity \(v\), and then the fixed curve \(\beta\),
both with twice the speed.

The map \(\Phi_\beta\) is affine and weakly continuous. Indeed, it has the form
\[
\Phi_\beta(v)=Lv+w,
\]
where
\[
(Lv)(t)
=
\begin{cases}
2v(2t), & t\in[0,\frac12],\\
0, & t\in[\frac12,1],
\end{cases}
\]
is a bounded linear operator on \(L^2([0,1],T_e\cG)\), and where
\[
w(t)
=
\begin{cases}
0, & t\in[0,\frac12],\\[0.4em]
2\,\partial_t\beta(2t-1)\circ\beta(2t-1)^{-1},
& t\in[\frac12,1],
\end{cases}
\]
is fixed.

For each \(n\), since \(u_n\in A_{x,[\alpha]}\), we have
\[
\operatorname{evol}(u_n)=x
\qquad\text{and}\qquad
\operatorname{Evol}(u_n)\in[\alpha].
\]
Thus \(\operatorname{Evol}(u_n)\) ends at \(x\), while \(\beta\) starts at
\(x\). Therefore the concatenation
\[
\operatorname{Evol}(u_n)*\beta
\]
is well-defined. Its Eulerian velocity is precisely \(\Phi_\beta(u_n)\), up
to the above affine reparametrisation. Consequently,
\[
\Phi_\beta(u_n)\in A_{e,[\alpha*\beta]}
\qquad\text{for all }n.
\]

Since \(A_{e,[\alpha*\beta]}\) is weakly closed by
condition~\ref{cond:weaklyclosed} in
\Cref{thm: periodic geodesic in all homotopy classes of loops}, and since
\(\Phi_\beta(u_n)\rightharpoonup \Phi_\beta(u)\), it follows that
\begin{equation}
\Phi_\beta(u)\in A_{e,[\alpha*\beta]}.
\end{equation}
In particular, \(\operatorname{evol}(\Phi_\beta(u))=e\). Thus it remains to
show that this implies \(\operatorname{evol}(u)=x\). For this we set
\[
y:=\operatorname{evol}(u).
\]
The first half of \(\operatorname{Evol}(\Phi_\beta(u))\) is
\(\operatorname{Evol}(u)\), reparametrised to the interval
\([0,\frac12]\), and hence it ends at \(y\). The second half of
\(\operatorname{Evol}(\Phi_\beta(u))\) has Eulerian velocity
\[
s\longmapsto \partial_s\beta(s)\circ\beta(s)^{-1},
\]
again up to the reparametrisation \(s=2t-1\), and starts at \(y\).

Since Eulerian velocities of the form
\(\partial_t\gamma\circ\gamma^{-1}\) are invariant under right translations,
this second half is the right translate of \(\beta\) given by
\[
\tilde\beta\colon [0,1]\to \cG,\qquad
s\longmapsto \beta(s)\circ x^{-1}\circ y.
\]
Indeed, a short computation confirms that this curve has starting point and
endpoint given by
\[
\tilde\beta(0)=y
\quad \text{and}\quad
\tilde\beta(1)=x^{-1}\circ y.
\]
Hence, using \(\operatorname{evol}(u)=y\), this can be expressed as
\[
\operatorname{evol}(\Phi_\beta(u))
=
x^{-1}\circ\operatorname{evol}(u).
\]
Recall that \(\Phi_\beta(u) \in A_{e,[\alpha*\beta]}\), so in particular
\(\operatorname{evol}(\Phi_\beta(u))=e\), and therefore
\[
x^{-1}\circ\operatorname{evol}(u)=e.
\]
Thus
\[
\operatorname{evol}(u)=x.
\]

Now the second half of \(\operatorname{Evol}(\Phi_\beta(u))\) is exactly
\(\beta\), not merely a right translate of \(\beta\). Hence
\[
\operatorname{Evol}(\Phi_\beta(u))
=
\operatorname{Evol}(u)*\beta
\]
up to the standard reparametrisation. Since
\(\Phi_\beta(u)\in A_{e,[\alpha*\beta]}\), this concatenated loop represents
the class \([\alpha*\beta]\). Cancelling the fixed path class \([\beta]\) in
the fundamental groupoid gives
\[
\operatorname{Evol}(u)\in[\alpha].
\]
Together with \(\operatorname{evol}(u)=x\), this shows
\[
u\in A_{x,[\alpha]}.
\]
Therefore \(A_{x,[\alpha]}\) is weakly closed.
\end{proof}
We are now in a position to prove that, under the assumptions of
\Cref{thm: periodic geodesic in all homotopy classes of loops}, the full
Hopf--Rinow statement also holds on the universal cover.
\begin{proposition}[Completeness Properties of the Universal Cover]\label{Lemm: universal cover geodesic convex}
  Let $(\cG,G)$ be as in \Cref{thm: periodic geodesic in all homotopy classes of loops}.  Then the universal 
cover $\widetilde{\cG}$ equipped with the induced right-invariant pull-back metric is metrically and geodesically complete. Furthermore, for any two points $\tilde{p}, \tilde{q}\in \widetilde{\cG}$ there exists a
  length–minimizing geodesic 
  $\tilde{\gamma}$ in $(\widetilde{\cG}, \tilde{G})$ connecting $\tilde{p}$ and $\tilde{q}$.
\end{proposition}

\begin{proof}
Since $\tilde\cG$ is again a Hilbert half-Lie group and $\tilde G$ is a strong and right-invariant metric, ~\Cref{thm: completeness result half lie groupd} implies that $(\tilde\cG, \tilde G)$ is metrically and geodesically complete.

  It remains to prove the existence of minimizing geodesics. For this we will use the strengthened weak closedness assumption, cf. Assumption~\ref{cond:weaklyclosed} of~\Cref{thm: periodic geodesic in all homotopy classes of loops}.
  Furthermore, it suffices to prove the result after
  right-translating by \(\tilde q^{-1}\). Thus we may assume without loss of
  generality that
  \[
    \tilde q=\tilde e,
    \qquad
    q=\pi(\tilde q)=e.
  \]
  Next, we set \(p:=\pi(\tilde p)\) and choose a smooth curve \(\tilde c:[0,1]\to\widetilde{\cG}\) with
  \[
    \tilde c(0)=\tilde e,
    \qquad
    \tilde c(1)=\tilde p,
  \]
  and let \(\alpha:=\pi\circ\tilde c\). Then \(\alpha\) is a path from \(e\)
  to \(p\). We denote its endpoint-homotopy class by \([\alpha]\), and set
  \[
    H^1_{e,p}([\alpha])
    :=
    \left\{
      \gamma\in H^1([0,1],\mathcal G)
      \,\middle|\,
      \gamma(0)=e,\ \gamma(1)=p,\ \gamma\in[\alpha]
    \right\}.
  \]
  Let \(\tilde\gamma_n\) be an energy-minimizing sequence of \(H^1\)-curves in
  \(\widetilde{\cG}\) connecting \(\tilde e\) and \(\tilde p\), that is,
  \[
    \tilde\gamma_n(0)=\tilde e,
    \qquad
    \tilde\gamma_n(1)=\tilde p,
  \]
  and
  \[
    \tilde E(\tilde\gamma_n)
    :=
    \frac12\int_0^1
    \tilde G_{\tilde\gamma_n}
    \bigl(\partial_t\tilde\gamma_n,\partial_t\tilde\gamma_n\bigr)
    \,\mathrm dt
    \longrightarrow
    \inf_{\tilde\gamma\in H^1_{\tilde e,\tilde p}([0,1],\widetilde{\cG})}
    \tilde E(\tilde\gamma).
  \]
  Set \(\gamma_n:=\pi\circ\tilde\gamma_n\).
  Since \(\pi\) is a local isometry, projection preserves energy. Moreover,
  by the path-lifting property, projection identifies curves from \(\tilde e\)
  to \(\tilde p\) with curves from \(e\) to \(p\) in the endpoint-homotopy
  class \([\alpha]\). Hence
  \begin{equation}\label{eq:E_gc_n_to_inf_alpha_E}
       E(\gamma_n)
    \longrightarrow
    \inf_{\gamma\in H^1_{e,p}([\alpha])}E(\gamma).
  \end{equation}
 For the right logarithmic velocity \(
    u_n:=\partial_t\gamma_n\circ\gamma_n^{-1} \) of $\gc_n$,  we have by using right invariance of $G$ and \(\gamma_n(0)=e\) that
  \[
    \gamma_n=\operatorname{Evol}(u_n),
    \qquad
    \operatorname{evol}(u_n)=p,
    \qquad
    \operatorname{Evol}(u_n)\in[\alpha].
  \]
  Thus \(u_n\in A_{p,[\alpha]}\). Since the energies $E(\gc_n)$ are by~\eqref{eq:E_gc_n_to_inf_alpha_E} bounded, after
  passing to a subsequence we may assume
  \[
    u_n\rightharpoonup u
    \quad\text{weakly in }L^2([0,1],T_e\cG).
  \]
  By \Cref{lem:weakly-closed-arbitrary-endpoints} the set \(A_{p,[\alpha]}\) is weakly closed, thus we obtain $u\in A_{p,[\alpha]}$.
 Set $\gamma:=\operatorname{Evol}(u)$. Then by $L^2$ regularity of $G$ we have \(\gamma\in H^1_{e,p}([\alpha])\), and thus by weak lower semicontinuity of
  the \(L^2\)-norm,
  \[
    E(\gamma)
    =
    \frac12\|u\|_{L^2}^2
    \le
    \liminf_{n\to\infty}\frac12\|u_n\|_{L^2}^2
    =
    \inf_{\eta\in H^1_{e,p}([\alpha])}E(\eta).
  \]
  Hence \(\gamma\) is energy-minimizing in the class \([\alpha]\). Therefore
  \(\gamma\) is a geodesic. By the usual constant-speed reparametrization
  argument, it is also length-minimizing in \([\alpha]\).

  Let \(\tilde\gamma\) be the unique lift of \(\gamma\) with \(\tilde\gamma(0)=\tilde e.\)
  Since \(\gamma\in[\alpha]\), and since \(\alpha\) lifts from \(\tilde e\) to
  \(\tilde p\), the path-lifting property gives $  \tilde\gamma(1)=\tilde p.$
  Since \(\pi\) is a local isometry, \(\tilde\gamma\) is a geodesic.

  Finally, let \(\tilde\eta\) be any \(H^1\)-curve from \(\tilde e\) to
  \(\tilde p\), and set \(\eta:=\pi\circ\tilde\eta\). Then
  \(\eta\in H^1_{e,p}([\alpha])\), and therefore
  \[
    L_{\tilde G}(\tilde\eta)
    =
    L_G(\eta)
    \ge
    L_G(\gamma)
    =
    L_{\tilde G}(\tilde\gamma).
  \]
  Thus \(\tilde\gamma\) is length-minimizing from \(\tilde e\) to
  \(\tilde p\). Translating back by \(\tilde q\) proves the result for the
  original points \(\tilde q\) and \(\tilde p\).
\end{proof}

Next we fix a nontrivial free homotopy class $[\alpha] \in \pi_1(\cG)$, corresponding to an element 
$\zeta \in \Gamma \setminus \{e\}$.  
Any loop $\eta : S^1 \to G$ in this class lifts to a path 
$\tilde{\eta} : [0,1] \to \widetilde G$ satisfying
\[
    \tilde{\eta}(0) = \tilde g_0, 
    \qquad 
    \tilde{\eta}(1) = \tilde g_0 \zeta.
\]
To prove~\Cref{thm: periodic geodesic in all homotopy classes of loops} we need one additional technical ingredient, which we state next:
\begin{lemma}[Positivity of the lifted minimum value.]\label{prop: minmx value homtopy class}
    If $[\alpha] \in \pi_1(\cG)$ is nontrivial, denote its associated element in 
    $\ker(\pi)$ by $\zeta$. Then
    \[
        c 
        \;:=\; 
        \inf_{\eta \in [\alpha]} \,E_G(\eta) 
        \;=\; 
        \frac12 d_{\widetilde G}(e,\zeta)^2 
        \;>\; 
        0,
    \]
    where $d_{\tilde{G}}$ denotes the geodesic distance on $(\widetilde \cG,\tilde{G})$.
\end{lemma}

\begin{proof}
    Write
    \[
        \mathcal{P}(e,\zeta)
        := 
        \bigl\{\, \sigma \in H^1([0,1], \widetilde \cG)
        \ \big|\ 
        \sigma(0)=e,\ \sigma(1)=\zeta \,\bigr\}.
    \]
    We claim that
    \begin{equation}\label{eq:inf-length=distance}
        \inf_{\eta\in[\alpha]} L_G(\eta) 
        \;=\; 
        \inf_{\sigma\in\mathcal{P}(e,\zeta)} L_{\tilde{G}}(\sigma)
        \;=\; 
        d_{\widetilde G}(e,\zeta),
    \end{equation}
    where $d_{\widetilde G}$ denotes the length metric on $(\widetilde \cG,\tilde G)$.

    Indeed, let $\eta : S^1 \to \cG$ be any loop in $[\alpha]$ and lift it to 
    $\tilde{\eta} : [0,1] \to \widetilde \cG$ with $\tilde{\eta}(0)=\tilde g_0$ and 
    $\tilde{\eta}(1)=\tilde g_0 \zeta$. This follows directly from covering space theory; since $\Gamma$ is central, the endpoint 
    of the lift is determined by the free homotopy class.

    Since $\pi$ is a local isometry, we have
    \[
        L_{G}(\eta)=L_{\tilde G}(\tilde{\eta}).
    \]
    Next we right-translate $\tilde{\eta}$ by $\tilde g_0^{-1}$ and define \( \sigma(t) := \tilde{\eta}(t)\,\tilde g_0^{-1}.\)
    Then $\sigma \in \mathcal{P}(e,\zeta)$.  
    Since right translation is an isometry of $(\widetilde \cG,\tilde{G})$, we have
    \[
        L_{\tilde G}(\sigma)=L_{\tilde G}(\tilde{\eta})=L_{G}(\eta).
    \]
    Taking infima over $\eta\in[\alpha]\in \pi_1(\cG)$ gives
    \[
        \inf_{\eta\in[\alpha]} L_G(\eta) 
        \;\ge\;
        \inf_{\sigma\in\mathcal{P}(e,\zeta)} L_{\tilde{G}}(\sigma).
    \]
    The reverse inequality follows by applying the same argument in the opposite direction by lifting any $\sigma \in \mathcal{P}(e,\zeta)$ down to a loop in $[\alpha]$ and translating it.  This proves \eqref{eq:inf-length=distance}. Combining \eqref{eq:inf-length=distance} with the standard fact from Riemannian geometry that
\[
   2 \inf_{\eta\in[\alpha]} E_G(\eta)
    \;=\;\bigl( \inf_{\eta\in[\alpha]} L_G(\eta) \bigr)^2,
\]
concludes the proof.
\end{proof}
We are finally ready to present the second proof of \Cref{thm: periodic geodesic in all homotopy classes of loops}:
\begin{proof}[Proof of \Cref{thm: periodic geodesic in all homotopy classes of loops} using the universal cover]
Let \([\alpha]\in \pi_1(\cG,e)\) be a nontrivial homotopy class. Via the standard identification \(\pi_1(\cG,e)\cong \ker(\pi)\),
the class \([\alpha]\) corresponds to a unique element
\(\zeta\in\ker(\pi)\setminus\{\widetilde e\}\). By
\Cref{Lemm: universal cover geodesic convex}, there exists a length-minimizing geodesic \(\widetilde\gamma:[0,1]\to\widetilde{\cG}\)
from \(\widetilde e\) to \(\zeta\). We choose \(\widetilde\gamma\) to be
parametrized with constant speed. Its projection
\[
\gamma:=\pi\circ\widetilde\gamma:[0,1]\to\cG
\]
is a closed loop in the homotopy class \([\alpha]\). Since \(\pi\) is a local
isometry and \(\widetilde\gamma\) is a geodesic, the curve \(\gamma\) is a
geodesic segment in \((\cG,G)\), that is closed. It remains to verify the tangent endpoint matching
condition
\[
\dot\gamma(0)=\dot\gamma(1).
\]
In order to do so we consider the Hilbert path space
\[
\mathcal P_\zeta
:=
\left\{
\sigma\in H^1([0,1],\widetilde{\cG})
\;:\;
\sigma(1)=\sigma(0)\zeta
\right\}.
\]
This is a smooth Hilbert submanifold of
\(H^1([0,1],\widetilde{\cG})\). Indeed, if
\[
\operatorname{ev}_{0,1}:
H^1([0,1],\widetilde{\cG})
\longrightarrow
\widetilde{\cG}\times\widetilde{\cG},
\qquad
\sigma\longmapsto (\sigma(0),\sigma(1)),
\]
denotes the endpoint map, then
\[
\mathcal P_\zeta
=
\operatorname{ev}_{0,1}^{-1}
\bigl(\operatorname{graph}(R_\zeta)\bigr),
\]
where \(R_\zeta(g)=g\zeta\). The endpoint map is a smooth submersion, and
\(\operatorname{graph}(R_\zeta)\) is a split embedded submanifold of
\(\widetilde{\cG}\times\widetilde{\cG}\), since \(R_\zeta\) is a smooth fixed
right translation. Note also that the tangent space of \(\mathcal P_\zeta\) at \(\widetilde\gamma\) is given by
\begin{equation}\label{eq:def_tangent_space_P_zeta}
    T_{\widetilde\gamma}\mathcal P_\zeta
=
\left\{
V\in H^1(\widetilde\gamma^*T\widetilde{\cG})
\;:\;
V(1)=T_{\widetilde e}R_\zeta V(0)
\right\}.
\end{equation}

We claim that \(\widetilde\gamma\) minimizes the energy \(E\) on
\(\mathcal P_\zeta\). Let \(\sigma\in\mathcal P_\zeta\) and set
\(\overline\sigma(t):=\sigma(t)\sigma(0)^{-1}\). Since
\(\sigma(1)=\sigma(0)\zeta\) and since \(\zeta\in Z(\widetilde{\cG})\), we have
\[
\overline\sigma(0)=\widetilde e,
\qquad
\overline\sigma(1)
=
\sigma(0)\zeta\sigma(0)^{-1}
=
\zeta .
\]
By right-invariance of \(\widetilde G\) combined with the fact that 
\(\widetilde\gamma\) is length minimizing from \(\widetilde e\) to
\(\zeta\) and is parametrized with constant speed, the Cauchy--Schwarz
inequality gives
\[
E(\widetilde\gamma)
=
\frac12 L_{\widetilde G}(\widetilde\gamma)^2
=
\frac12 d_{\widetilde G}(\widetilde e,\zeta)^2
\leq
\frac12 L_{\widetilde G}(\overline\sigma)^2
\leq
E(\overline\sigma)
=
E(\sigma).
\]
Thus \(\widetilde\gamma\) is a minimizer of \(E\) on \(\mathcal P_\zeta\). Therefore we have
\begin{equation}\label{eq:crit_point_P_zeta}
    \mathrm d(E|_{\mathcal P_\zeta})_{\widetilde\gamma}(V)=0
\qquad
\forall
V\in T_{\widetilde\gamma}\mathcal P_\zeta .
\end{equation}
By the first variation formula for the energy $E$ on the space of paths with moving endpoints, that is $\cP_\zeta$, we have
\[
\mathrm dE_{\widetilde\gamma}(V)
=
\widetilde G_\zeta
\left(
\dot{\widetilde\gamma}(1),
V(1)
\right)
-
\widetilde G_{\widetilde e}
\left(
\dot{\widetilde\gamma}(0),
V(0)
\right)
-
\int_0^1
\widetilde G_{\widetilde\gamma(t)}
\left(
\nabla_t\dot{\widetilde\gamma},
V
\right)
\,\mathrm dt .
\]
Since \(\widetilde\gamma\) is a geodesic of $(\widetilde \cG, \widetilde G)$, the integral term on the right-hand side of the equation above vanishes. This in combination with~\eqref{eq:crit_point_P_zeta} yields
\begin{equation}\label{eq:tilde_G_dot_gc_V_equal_zero}
    0
=
\widetilde G_\zeta
\left(
\dot{\widetilde\gamma}(1),
V(1)
\right)
-
\widetilde G_{\widetilde e}
\left(
\dot{\widetilde\gamma}(0),
V(0)
\right)\qquad \forall V\in T_{\widetilde\gamma}\mathcal P_\zeta.
\end{equation}
For a fixed \(X\in T_{\widetilde e}\widetilde{\cG}\) we choose
\(V\in T_{\widetilde\gamma}\mathcal P_\zeta\) so that \(V(0)=X\). Then by~\eqref{eq:def_tangent_space_P_zeta} it holds that \(V(1)=T_{\widetilde e}R_\zeta X
\).
Therefore by~\eqref{eq:tilde_G_dot_gc_V_equal_zero}
\[
0
=
\widetilde G_\zeta
\left(
\dot{\widetilde\gamma}(1),
T_{\widetilde e}R_\zeta X
\right)
-
\widetilde G_{\widetilde e}
\left(
\dot{\widetilde\gamma}(0),
X
\right).
\]
As this holds for all \(X\in T_{\widetilde e}\widetilde{\cG}\), and since
\(R_\zeta\) is an isometry, we conclude that
\[
\dot{\widetilde\gamma}(1)
=
T_{\widetilde e}R_\zeta
\,\dot{\widetilde\gamma}(0).
\]
For the projection \(\gamma=\pi\circ\widetilde\gamma\), using
\(\pi\circ R_\zeta=\pi\), we obtain
\[
\dot\gamma(1)
=
T_\zeta\pi\,\dot{\widetilde\gamma}(1)
=
T_\zeta\pi\,
T_{\widetilde e}R_\zeta
\,\dot{\widetilde\gamma}(0)
=
T_{\widetilde e}\pi\,
\dot{\widetilde\gamma}(0)
=
\dot\gamma(0).
\]
Hence \(\gamma\) is \(C^1\) as a loop and therefore extends to a smooth
periodic geodesic of \((\cG,G)\) with domain of definition \(S^1\).

By \Cref{prop: minmx value homtopy class}, the length of \(\gamma\) is
\(d_{\widetilde G}(\widetilde e,\zeta)\). Thus \(\gamma\) is a nontrivial
periodic geodesic of \((\cG,G)\) satisfying \([\gamma]=[\alpha]\). As in the
first proof, the standard iteration argument shows that if the homotopy class
\([\alpha]\) is primitive, then \(\gamma\) is a prime geodesic of
\((\cG,G)\). This completes the proof.
\end{proof}

\subsection{A Lyusternik–-Fet Theorem for Non-Aspherical Half-Lie Groups satisfying a Palais--Smale condition}
A classical problem in Riemannian geometry asks whether, in addition to the existence of a closed geodesic in every homotopy class, there also exists a periodic geodesic that is contractible. In the compact, finite-dimensional setting, the famous Lyusternik–Fet theorem answers this question affirmatively. We now address the analogous problem for half-Lie groups under the assumption that the energy functional  satisfies a Palais–Smale condition modulo right translations. Therefore, we will first note that  the group \(\cG\) acts on the free loop space $\Lambda \cG$ by right translations, i.e., for $\gamma\in \Lambda \cG$ and $g\in\cG$
\[
(\gamma,g)\mapsto \gamma g,
\qquad
(\gamma g)(t):=\gamma(t)g.
\]
Since the metric $G$ is right-invariant, the energy functional $E$
is invariant under this action, i.e., 
$E(\gamma g)=E(\gamma)$. This gives rise to the following notion.

\begin{definition}[\((PS)_c\) modulo right translations]
Let \(c\in\bR\)\label{def:ps_c_mod_right_trans}. We say that the energy functional \(E\) satisfies the
Palais--Smale condition at level \(c\) modulo right translations on $\Lambda_0\cG$ if every
\((PS)_c\)-sequence \((\gamma_n)\subset\Lambda_0\cG\) for \(E\), in the sense of
\Cref{def:def_and_cond_PS}, admits a sequence \((g_n)\subset\cG\) such that,
after passing to a subsequence,
\[
\gamma_n g_n\to\gamma
\qquad
\text{in } \Lambda_0\cG.
\]
\end{definition}

\begin{remark}[Non-right invariant \((PS)_c\) condition]
We emphasize that, whenever \(\cG\) is noncompact, the energy functional
\(E\) itself can never satisfy the ordinary $(PS)_c$-condition on
\(\Lambda\cG\). Indeed, if \(\gamma\) is a critical point of \(E\), then so is
every right translate \(\gamma g\), \(g\in\cG\), and
\[
E(\gamma g)=E(\gamma).
\]
Choosing a sequence \(g_n\to\infty\) in \(\cG\), the sequence
\[
(\gamma g_n)
\]
is a \((PS)_{E(\gamma)}\)-sequence without any convergent subsequence in
\(\Lambda\cG\). Consequently, the reduced Palais--Smale condition modulo right translations for
\(E\) is the natural compactness assumption in the right-invariant
setting.
\end{remark}
\begin{remark}[\((PS)_c\) condition of the reduced energy functional]
We note that, since \(E\) is invariant under fixed right translations, it descends to a
well-defined function on the orbit set
\begin{equation}\label{eq:reduced_energy}
 \bar E:\Lambda\cG/\cG\to\mathbb R,
 \qquad
 \bar E([\gamma]) := E(\gamma),
\end{equation}
where \(\Lambda\cG/\cG\) denotes the set of orbits of the right-translation
action. At this point no smooth structure on the quotient is being assumed. For every fixed \(g\in\cG\), the right translation
\[
R_g:\cG\to\cG,
\qquad
x\mapsto xg,
\]
is smooth, and hence induces a smooth map on the loop space,
\[
\Lambda\cG\to\Lambda\cG,
\qquad
\gamma\mapsto \gamma g.
\]
This is the only smoothness used in the definition of the Palais--Smale
condition modulo right translations. On the contractible component, there is a canonical identification at the level
of sets,
\[
\Lambda_0\cG/\cG\cong\Omega_e\cG,
\qquad
[\gamma]\longmapsto \gamma\gamma(0)^{-1}.
\]
Indeed, for each fixed loop \(\gamma\), right translation by the fixed element
\(\gamma(0)^{-1}\) sends \(\gamma\) to a based loop. However, in the half-Lie
group setting, one should not automatically regard the map
\[
\gamma\longmapsto \gamma\gamma(0)^{-1}
\]
as a smooth map on the loop space, since this would require smooth dependence
on the translating element \(\gamma(0)^{-1}\), which is not part of the
half-Lie group structure. Consequently, the Palais--Smale condition modulo right translations, as introduced in~\Cref{def:ps_c_mod_right_trans}, should be
understood as a compactness condition formulated directly on
\(\Lambda_0\cG\), not automatically as the usual Palais--Smale condition for a
smooth reduced functional. If, in addition, \(\Lambda_0\cG/\cG\) carries a
compatible smooth structure for which \(\bar E\) is smooth and quotient
convergence is represented by convergence after right translations, then the
two formulations agree. Without such additional structure, this equivalence
should not be assumed.
\end{remark}

In order to formulate the Lyusternik–-Fet theorem, we first
recall, for the convenience of the reader, how a non-vanishing higher homotopy
group gives rise to a minimax value. This minimax value will turn out to be
positive; this positivity is a nontrivial point and is not automatic for
arbitrary minimax constructions, even on non-compact finite-dimensional manifolds.

Assume that there exists \(q\geq 2\) such that \(\pi_q(\cG)\neq 0\). Using the
standard identification
\[
\pi_q(\cG)\cong \pi_{q-1}(\Omega_0\cG),
\]
where \(\Omega_0\cG\) denotes the space of based contractible \(H^1\)-loops in
\(\cG\), we obtain a nontrivial homotopy class
\[
[\Gamma]\in \pi_{q-1}(\Omega_0\cG)\setminus\{0\}.
\]
We regard \(\Omega_0\cG\) as a subspace of the contractible component
\(\Lambda_0\cG\) of the free loop space. Thus the class \([\Gamma]\) gives rise
to a homotopy class in \(\Lambda_0\cG\), which we denote again by
\([\Gamma]\). More explicitly, we define
\begin{equation}\label{eq:defn_Gamma_half_lie_group}
    \Gamma :=
    \left\{
    \gamma \in C^0(S^{q-1},\Lambda_0\cG)
    \;\middle|\;
    [\gamma]=[\Gamma] \text{ in } \pi_{q-1}(\Lambda_0\cG)
    \right\}.
\end{equation}
The minimax value associated with \([\Gamma]\) is then defined by
\begin{equation}\label{eq:minimax_value_higher_homotopy_group_half_lie_group}
    c(\Gamma):=
    \inf_{\gamma\in\Gamma}\sup_{\xi\in S^{q-1}} E(\gamma(\xi)).
\end{equation}
We are now able to formulate the Lyusternik--Fet theorem:
\begin{theorem}[Lyusternik–-Fet theorem for half-Lie groups]
\label{thm:Lyusternik_half}
Let \((\cG,G)\) be a half-Lie group equipped with a strong
right-invariant Riemannian metric. Assume, in addition, that
\begin{enumerate}[label=(\alph*)]
\item\label{it:1 Lyusternik_half_homotopy} there exists \(q\ge2\) such that
\(
\pi_q(\cG)\neq0,
\) and that,
\item\label{it:2 Lyusternik_half_rigth_PS_c}  for the induced nontrivial class
\[
0\neq[\Gamma]\in\pi_{q-1}(\Lambda_0\cG),
\]
the energy functional $E$ satisfies the Palais--Smale condition at the
level $c(\Gamma)$ as in \eqref{eq:minimax_value_higher_homotopy_group_half_lie_group} modulo right translations.
\end{enumerate}
Then \((\cG,G)\) admits a contractible, nonconstant  periodic geodesic of energy $c(\Gamma)>0$.
\end{theorem}
\begin{comment}
    
\begin{theorem}[Lyusternik--Fet theorem for half-Lie groups]
\label{thm:Lyusternik_half}
Let \((\cG,G)\) be a half-Lie group equipped with a strong
right-invariant Riemannian metric. Assume, in addition, that
\begin{enumerate}[label=(\alph*)]
\item there exists \(q\ge2\) such that
\(
\pi_q(\cG)\neq0,
\) and that,
\item for the induced nontrivial class
\[
0\neq[\Gamma]\in\pi_{q-1}(\Lambda_0\cG/\cG),
\]
the reduced energy functional $\bar E$, as introduced in~\eqref{eq:reduced_energy}, satisfies the Palais--Smale condition at the
level
\begin{equation}\label{eq:minmax}
    c(\Gamma)
:=
\inf_{\gamma\in[\Gamma]}
\sup_{\xi\in S^{q-1}}
\bar E(\gamma(\xi)).
\end{equation}
\end{enumerate}
Then \((\cG,G)\) admits a contractible,  periodic geodesic.
\end{theorem}
\end{comment}

\begin{remark}[Alternative compactness assumption]
Alternatively, instead of assuming the full Palais--Smale condition for
\(\bar E\) at the level \(c(\Gamma)\), one may impose the corresponding
compactness assumption from
\Cref{thm: periodic geodesic in all homotopy classes of loops}. Then, following
the argument of the direct-method proof of
\Cref{thm: periodic geodesic in all homotopy classes of loops}, it suffices to
assume the following norm convergence property.

Let \((\gamma_n)\) be the \((PS)_{c(\Gamma)}\)-sequence constructed in
\Cref{lemm:Existence_of Palais_Smale_sequence}, and let
\[
u_n := \partial_t\gamma_n\circ\gamma_n^{-1}
\in L^2([0,1],T_e\cG)
\]
be its right logarithmic derivatives. If, after passing to a subsequence,
\[
u_n \rightharpoonup u
\qquad\text{in } L^2([0,1],T_e\cG),
\]
then one assumes that
\[
\|u_n\|_{L^2([0,1],T_e\cG)}
\to
\|u\|_{L^2([0,1],T_e\cG)} .
\]
This implies
\[
u_n\to u
\qquad\text{strongly in } L^2([0,1],T_e\cG),
\]
and hence, by \(L^2\)-regularity,
\[
\gamma_n\to\gamma:=\Evol(u)
\qquad\text{in } H^1([0,1],\cG).
\]

We note that, without such a norm convergence assumption, it is unclear whether
one can exclude loss of norm along weakly convergent subsequences of
\((u_n)\).
\end{remark}

%Having formulated the Palais--Smale condition modulo right translations, we
%also record the corresponding \(q=1\) statement. The existence of periodic
%geodesics in prescribed nontrivial homotopy classes was proved above by the
%direct method in
%\Cref{thm: periodic geodesic in all homotopy classes of loops}. The following
%result gives an alternative Palais--Smale point of view: instead of using the
%direct method, one assumes compactness directly at the relevant minimax level.
%The proof follows the same minimax argument as in \Cref{thm:Lyusternik_half}.

% \begin{theorem}[Periodic geodesics from a \(PS_c\)-condition modulo right translations]
% \label{thm:periodic_geodesic_from_PS_mod_right_translations}
% Let \((\cG,G)\) be a half-Lie group equipped with a strong right-invariant
% Riemannian metric. Assume that:
% \begin{enumerate}[label=(\alph*)]
%     \item\label{it:pi1_nontrivial_PS_theorem}
%     \(\pi_1(\cG)\neq0\);

%     \item\label{it:PS_level_homotopy_class}
%     for some nontrivial homotopy class \([\alpha]\in\pi_1(\cG)\setminus\{0\}\), the energy
%     functional \(E\) satisfies the Palais--Smale condition at the level
%     \(c([\alpha])\) modulo right translations, where \(c([\alpha])\) denotes the
%     minimum value from \eqref{eq:min_energy_value_homotopy_class}.
% \end{enumerate}
% Then \((\cG,G)\) admits a nonconstant periodic geodesic of energy
% \(c([\alpha])>0\) in the homotopy class \([\alpha]\).
% \end{theorem}

\begin{remark}[Palais--Smale approach in a nontrivial homotopy class]
The existence result for $\pi_1(\cG)\neq 0$ of~\Cref{thm: periodic geodesic in all homotopy classes of loops} could also be obtained from a
Palais--Smale compactness assumption, in analogy with the proof of
Theorem~\ref{thm:Lyusternik_half}. Indeed, let
\[
 c([\alpha])
=
\inf_{\gamma\in[\alpha]} E(\gamma)
\]
for a nontrivial homotopy class \([\alpha]\in\pi_1(G)\setminus\{0\}\).
Since \(\Lambda_{[\alpha]}G\) is a connected component of the complete loop space
\(\Lambda G\), Ekeland's variational principle yields a
\((PS)_{c([\alpha])}\)-sequence in \(\Lambda_{[\alpha]}G\).
If, in addition, the energy functional satisfies the Palais--Smale condition
modulo right translations at the level \(c([\alpha])\), then one obtains a
critical point \(\gamma\in\Lambda_{[\alpha]}G\) with
\[
E(\gamma)=c([\alpha]).
\]
Since we have already shown that \(c([\alpha])>0\),
the loop \(\gamma\) is nonconstant and therefore defines a periodic geodesic
representing the class \([\alpha]\).
\end{remark}

% \begin{remark}
% This should 
%be read as a Palais--Smale reformulation of the
%\(q=1\) minimax argument, 
% not be viewed as a replacement for the direct-method result.
% Compared with
% \Cref{thm: periodic geodesic in all homotopy classes of loops}, the compactness
% input provided by condition~\ref{cond:weaklyclosed} is shifted to
% condition~\ref{it:PS_level_homotopy_class} in
% \Cref{thm:periodic_geodesic_from_PS_mod_right_translations}. In particular, the
% structural assumptions \ref{cond:L_2 regular} and \ref{cond:weaklyclosed}  in \Cref{thm: periodic geodesic in all homotopy classes of loops} are
% not part of the statement above, although the authors believe that they are expected to be relevant when
% verifying this Palais--Smale condition in concrete examples.
%\end{remark}

We now return to the higher-homotopy setting of \Cref{thm:Lyusternik_half}.
Before proving \Cref{thm:Lyusternik_half}, we record the following consequence.
By an argument following the lines of
\Cref{cor:infinitely_many_periodic_geodeisc_in_each_homotopy_class}, the
existence of one contractible periodic geodesic in the infinite-dimensional
setting implies the existence of infinitely many geometrically distinct
contractible periodic geodesics.

\begin{corollary}[Infinitely many geometrically distinct contractible periodic geodesics]
\label{cor:infinitely_many_contractible_periodic_geodesics}
Let \((\cG,G)\) be as in \Cref{thm:Lyusternik_half}. If \(\cG\) is
infinite-dimensional, then there exist infinitely many geometrically distinct
contractible periodic geodesics of \((\cG,G)\).
\end{corollary}

As already pointed out before stating the main theorem of this subsection, the
proof of \Cref{thm:Lyusternik_half} is based on Morse theory for the energy
functional on the space of contractible loops, using minimax techniques.
Therefore, we first show in the following lemma that the minimax value
\eqref{eq:minimax_value_higher_homotopy_group_half_lie_group} is strictly
positive.
\begin{lemma}[Positivity of the minimax value]
Let \((\cG,G)\) be as in~\Cref{thm:Lyusternik_half}. Then
$c(\Gamma)$, as defined in~\eqref{eq:minimax_value_higher_homotopy_group_half_lie_group}, is strictly positive.
\end{lemma}
\begin{proof}
Since the metric \(G\) is strong, there exists a geodesically convex
neighbourhood \(U\subset \cG\) of the identity \(e\). Choose \(r>0\)
sufficiently small such that
\[
B_r(e)\subset U
\]
is geodesically convex. By right-invariance of \(G\), every right
translation \(x\mapsto x\cdot g\) is an isometry. Hence, for every \(g\in\cG\), the ball \(B_r(g)=R_g(B_r(e))\) is geodesically convex.

Choose \(\varepsilon>0\) sufficiently small such that \(\sqrt{2\varepsilon}<r\). Each contractible loop of energy less than \(\varepsilon\), that is,
\[
\alpha\in \Lambda_0\cG
\qquad\text{with}\qquad
E(\alpha)<\varepsilon,
\]
is completely contained in a metric ball,
\[
\alpha([0,1])\subset B_r(\alpha(0)),
\]
which is geodesically convex. Hence every such loop
\(\alpha\) can be contracted to the constant loop at \(\alpha(0)\)
through loops remaining entirely inside \(B_r(\alpha(0))\). It follows that the inclusion
\begin{equation}\label{eq:inclusion_lamba_eps_is_homotopic_to_const}
\Lambda_0^{<\varepsilon}\cG
:=
\{\alpha\in\Lambda_0\cG\mid E(\alpha)<\varepsilon\}
\hookrightarrow \Lambda_0\cG
\end{equation}
is homotopic to the map
\begin{equation}\label{eq:s_ev_0}
s\circ \operatorname{ev}_0:
\Lambda_0^{<\varepsilon}\cG\to \Lambda_0\cG,
\end{equation}
where the evaluation map is denoted by
\[
\operatorname{ev}_0:\Lambda_0\cG\longrightarrow\cG,\qquad \gamma\mapsto\gamma(0),
\]
and \(s:\cG\to\Lambda_0\cG\) denotes the inclusion of constant loops.
Assume now, by contradiction, that
\[
c(\Gamma)=0.
\]
Then, by definition of the minimax value \(c(\Gamma)\) in~\eqref{eq:minimax_value_higher_homotopy_group_half_lie_group}, for every
\(\delta>0\) there exists a representative \(\gamma:S^{q-1}\to\Lambda_0\cG\)
of the higher homotopy class \([\Gamma]\) such that
\[
\sup_{\xi\in S^{q-1}} E(\gamma(\xi))<\delta.
\]
From this we can conclude that
\[
\gamma(S^{q-1})\subset \Lambda_0^{<\delta}\cG.
\]
Choosing \(\delta<\varepsilon\), we obtain
\[
\gamma:S^{q-1}\to \Lambda_0^{<\varepsilon}\cG.
\]
By \eqref{eq:inclusion_lamba_eps_is_homotopic_to_const} and \eqref{eq:s_ev_0}, this map is homotopic in
\(\Lambda_0\cG\) to \(s\circ \operatorname{ev}_0\circ \gamma\). Therefore,
\begin{equation}\label{eq:Gamma_equal_s_ev_0}
[\Gamma]
=
[\gamma]
=
s_*(\operatorname{ev}_0)_*[\gamma].
\end{equation}
By definition of \(\Gamma\) in \eqref{eq:defn_Gamma_half_lie_group}, there exists a homotopy class
\([\Gamma_0]\in \pi_{q-1}(\Omega_e\cG)\) such that, for the inclusion
\(i:\Omega_e\cG\hookrightarrow \Lambda_0\cG\), it holds that
\[
[\Gamma]=i_*[\Gamma_0].
\]
Together with \eqref{eq:Gamma_equal_s_ev_0}, this implies that
\[
(\operatorname{ev}_0)_*[\gamma]
=
(\operatorname{ev}_0\circ i)_*[\Gamma_0].
\]
Since
\[
\operatorname{ev}_0\circ i:\Omega_e\cG\to\cG
\]
is the constant map with value \(e\), we obtain
\[
(\operatorname{ev}_0\circ i)_*[\Gamma_0]=0 \quad \text{in } \pi_{q-1}(\cG).
\]
Thus \((\operatorname{ev}_0)_*[\gamma]=0\) in \(\pi_{q-1}(\cG)\), and consequently, by \eqref{eq:Gamma_equal_s_ev_0},
\[
[\Gamma]
=
s_*(\operatorname{ev}_0)_*[\gamma]
=
s_*(0)
=
0,
\]
which contradicts the assumption that \([\Gamma]\neq 0\) in
\(\pi_{q-1}(\Lambda_0\cG)\). Therefore
\[
c(\Gamma)>0.
\]
This finishes the proof.
\end{proof}
Next, we show that the Riemannian energy functional admits a Palais--Smale sequence at the level $c(\Gamma)$.

\begin{lemma}[Existence of a Palais--Smale sequence]\label{lemm:Existence_of Palais_Smale_sequence}
Let $(\cG,G)$ be a half-Lie group equipped with a strong right-invariant metric $G$, and assume that there exists $q \geq 2$ such that $\pi_q(G)\neq 0$. Then, for every nontrivial class
\[
\Gamma\in \pi_{q-1}(\Lambda_0\cG)\setminus\{0\},
\]
the Riemannian energy functional $E$ admits a Palais--Smale sequence at level $c(\Gamma)$, where $c(\Gamma)$ is defined in \eqref{eq:minimax_value_higher_homotopy_group_half_lie_group}.
\end{lemma}
\begin{proof}
We want to make use of \Cref{thm:General_Minimax_Principle}. Thus, we have to check its prerequisites. Consider the Riemannian energy functional $E$ induced by the strong right-invariant metric $G$ on the loop space $\Lambda(\cG)$. This functional is smooth, and by \Cref{thm: completeness result half lie groupd}, the space $(\cG,G)$ is metrically complete. Thus, by \Cref{thm:loop_space_is_complete}, the loop space \(\left(\Lambda(\cG), \langle\cdot, \cdot\rangle_{H^1}\right)\), equipped with the $H^1$-metric induced by $G$, is metrically complete as well.
 
By using \Cref{thm:loop_space_is_complete} again, we can conclude that, since $G$ is strong, the induced $H^1$-metric on $\Lambda(\cG)$ is strong too. Thus, there exists a unique gradient vector field $\nabla E$ of $E$ on \(\left(\Lambda(\cG), \langle\cdot, \cdot\rangle_{H^1}\right)\). By the Picard--Lindelöf theorem, the negative gradient flow, denoted by $\varPhi_E^t$, of $E$ exists locally in time. It leaves $\Lambda_0(\mathcal G)$, the path-connected component of the contractible loops, invariant. Indeed, for each $t$, the map $\Phi_E^t$ is continuous, hence it sends path-connected sets to path-connected sets. Therefore, a loop starting in $\Lambda_0(\mathcal G)$ remains in the same path connected component for as long as the flow exists. 

It remains to check that the homotopy class $[\Gamma]$, as in~\eqref{eq:defn_Gamma_half_lie_group}, is positively invariant under $\varPhi_E$. Therefore, let $\gamma\in\Gamma$, that is, by \eqref{eq:defn_Gamma_half_lie_group}, a $C^0$-map \[\gamma:S^{q-1}\longrightarrow\Lambda_0\cG\] such that $[\gamma]=[\Gamma]\in \pi_{q-1}(\Lambda_0\cG)$, and let $T> 0$ be such that $\varPhi_E^T$ is defined on $\gamma(S^{q-1})$. Since $\varPhi_E^T\colon \Lambda_0\cG\to \Lambda_0\cG$ is continuous, the composition
\[
\varPhi_E^T\circ \gamma \colon S^{q-1}\longrightarrow \Lambda_0\cG
\]
is again continuous. Moreover, by the definition of $\varPhi_E$ it holds that $\varPhi_E^0=\mathrm{id}_{\Lambda_0\cG}$, and the map
\[
[0,T]\times \Lambda_0\cG\longrightarrow \Lambda_0\cG,\qquad (t,\alpha)\mapsto \varPhi_E^t(\alpha),
\]
is continuous. Hence
\[
H\colon [0,T]\times S^{q-1}\longrightarrow \Lambda_0\cG,\qquad (t,\xi)\mapsto H(t,\xi):=\varPhi_E^t(\gamma(\xi)),
\]
defines a homotopy between $\gamma$ and $\varPhi_E^T\circ\gamma$. Therefore,
\[
[\varPhi_E^T\circ\gamma]=[\gamma]=[\Gamma]
\qquad\text{in }\pi_{q-1}(\Lambda_0\cG).
\]
This shows that $\varPhi_E^T\circ\gamma\in\Gamma$, and thus $\Gamma$ is positively invariant under the negative gradient flow.

We have therefore verified the assumptions of \Cref{thm:General_Minimax_Principle}. Consequently, there exists a Palais--Smale sequence $(\gamma_n)\subseteq \Lambda_0\cG$ of $E$ at level $c(\Gamma)>0.$
\end{proof}
Now we are in the position to finish the proof of \Cref{thm:Lyusternik_half}. 
\begin{proof}[Proof of \Cref{thm:Lyusternik_half}]
By \Cref{lemm:Existence_of Palais_Smale_sequence}, the Riemannian energy functional $E$ admits a Palais--Smale sequence $(\gamma_n)\subseteq \Lambda_0\cG$ at level $c(\Gamma)$, where $c(\Gamma)$ is defined in \eqref{eq:minimax_value_higher_homotopy_group_half_lie_group}. By condition \ref{it:2 Lyusternik_half_rigth_PS_c} in \Cref{thm:Lyusternik_half}, the energy functional $E$ satisfies the $(PS)_c$ condition at level $c(\Gamma)$ modulo right translations in the sense of \Cref{def:ps_c_mod_right_trans}. 

Hence, there exists a sequence $(g_n)\subset \cG$ so that after passing to a subsequence, $(\gamma_n\cdot g_n)$ converges in $\Lambda_0\cG$ to a critical point $\gamma$ of $E$ with
\[
E(\gamma)=c(\Gamma).
\]
Since $c(\Gamma)>0$ and $\dd E(\gamma)=0$, the loop $\gamma$ is a nonconstant contractible periodic geodesic of $(\cG,G)$.
\end{proof}
\subsection{Reduction to Finite-Dimensional Totally Geodesic Submanifolds}
We conclude this section by describing a particular simple situation, in which the existence of periodic geodesics can be reduced to classical finite-dimensional results. In the following, we do not require the metric to be strong and therefore permit weak Riemannian metrics, provided that the Levi-Civita connection exists. Note that the existence of the Levi-Civita connection  is automatically satisfied for strong Riemannian metrics, but that it may fail in the weak Riemannian setting~\cite{bauer2014homogeneous}.
\begin{proposition}\label{prop_dyn_reduction}
    Let $(\mathcal{G}, G)$ be a half-Lie group equipped
with a right-invariant metric $G$, such that the Levi--Civita connection $\nabla$ associated with $(\mathcal{G}, G)$ exists. 
   Assume, in addition, that there exists a compact finite-dimensional totally geodesic submanifold $(\bar{\mathcal{M}}, \bar{G})$ of $(\mathcal{G}, G)$. We have:
    \begin{enumerate}
        \item\label{it:1_prop_dyn_reduction} 
        If $\pi_1(\bar\cM) \neq 0$, then for each nontrivial homotopy class 
        $[c] \in \pi_1(\bar\cM) \setminus \{0\}$ there exists a closed geodesic of $(\mathcal{G}, G)$, 
        which is prime if the homotopy class $[c]$ is primitive. 
        
        \item\label{it:2_prop_dyn_reduction} 
        If there exists $q \geq 2$ such that $\pi_q(\bar{\mathcal{M}}) \neq 0$, 
        then there exists a nonconstant closed geodesic of $(\cG, G)$.
    \end{enumerate}
Note that if $\cG$ is infinite-dimensional, then in each of the two situations the right-invariance of the metric $G$ implies the existence of infinitely many geometrically distinct periodic geodesics. 
\end{proposition}
\begin{proof}
This follows directly from the classical finite-dimensional results~\cite{Klingenberg1978, LystFetThm51}, since periodic geodesics in the totally geodesic submanifold give rise to periodic geodesics in the infinite-dimensional half-Lie group. By an argument analogous to \Cref{cor:infinitely_many_periodic_geodeisc_in_each_homotopy_class}, we can conclude the existence of infinitely many geometrically distinct periodic geodesics.
\end{proof}
\section{Examples of half-Lie groups with periodic geodesics}\label{sec:examples}
In this section, we illustrate the applicability of \Cref{thm: periodic geodesic in all homotopy classes of loops}, \Cref{thm:Lyusternik_half}, and \Cref{prop_dyn_reduction} through a collection of examples of half-Lie groups satisfying the respective hypotheses. We first show that groups of Sobolev diffeomorphisms equipped with suitable right-invariant metrics satisfy the assumptions of \Cref{thm: periodic geodesic in all homotopy classes of loops} and \Cref{prop_dyn_reduction}. By contrast, the validity of the Palais–Smale-type hypothesis required in \Cref{thm:Lyusternik_half} is presently unknown for these examples. Finally, in \Cref{sec:ex:synthetic}, we construct a synthetic but genuinely infinite-dimensional example for which all assumptions of \Cref{thm:Lyusternik_half} can be verified.

\subsection{Groups of Sobolev Diffeomorphisms Equipped with Right-Invariant Strong Metrics}
As a first example we will study strong right-invariant metrics on $\Diff^s(M)$ as introduced in~\Cref{ssec:back:diffeo}. In this setting we can immediately apply~\Cref{thm: periodic geodesic in all homotopy classes of loops} to obtain the following existence result for periodic geodesics:
\begin{corollary}\label{cor:infinitely many periodic geodesics sobolev diffeos}
Let $s>\frac{\operatorname{dim}(M)}2+1$, let $L$ be a  positive, elliptic, self-adjoint, pseudodifferential operators  of order $2s$ and let \(\Diff^s(M)\) be equipped with the right-invariant metric \(G^L\) induced by $L$ as in
\eqref{eq:def_H_r_metric_id}. Assume, in addition, that \[\pi_1(\Diff^s(M))\neq 0.\] Then in each nontrivial homotopy class $[\alpha]\in \pi_{1}(\Diff^s(M))$ there exist infinitely many geometrically distinct nonconstant 
periodic geodesics of \((\Diff^s(M),G^L)\) in the homotopy class $[\alpha]$.
%The results continue to hold for the subgroups of volume preserving diffeomorphisms $\Diff^s_{\mu}(M)$ and symplectomorphisms $\on{Sympl}^s(M,\omega)$.
\end{corollary}
\begin{remark}
    The condition \(\pi_1(\Diff^s_0(M))\neq 0\) holds for the following non-exclusive list of manifolds:
    \begin{enumerate}
        \item for \(M=S^1\), as \(\Diff^s_0(S^1)\) is homotopy equivalent to \(S^1\);

        \item for \(M=S^2\), as \(\Diff^s_0(S^2)\) is, by Smale's theorem~\cite{Smale1959}, homotopy equivalent to \(SO(3)\);

        \item for \(M=T^n\), \(n\geq 1\). Indeed, translations give an inclusion
        \[
            T^n \hookrightarrow \Diff^s_0(T^n),
        \]
        and evaluation at a point gives a left inverse. Hence
        \[
            \mathbb Z^n \cong \pi_1(T^n)\hookrightarrow \pi_1(\Diff^s_0(T^n)).
        \]

        \item for total spaces \(M\) of principal \(S^1\)-bundles
        \[
            S^1\hookrightarrow M\to B
        \]
        for which the fiber represents a nontrivial element of \(\pi_1(M)\). The fiberwise rotation action gives a loop in \(\Diff^s_0(M)\), and evaluation at a point detects the fiber class. Thus
        \[
            \pi_1(\Diff^s_0(M))\neq 0.
        \]
        For instance, the circle bundles
        \[
            S^1\hookrightarrow L(k,1)\to S^2,\qquad k\geq 2,
        \]
        give such examples, with the fiber class of order \(k\). The case \(k=1\) is the Hopf fibration
        \[
            S^1\hookrightarrow S^3\to S^2,
        \]
        but here the fiber class is trivial in \(\pi_1(S^3)\); the nontriviality of
        \(\pi_1(\Diff^s_0(S^3))\) instead follows from Hatcher's proof of the Smale conjecture~\cite{Hatcher1983}.
    \end{enumerate}
\end{remark}
\begin{proof}[Proof of~\Cref{cor:infinitely many periodic geodesics sobolev diffeos}]
One has to check the assumptions of
\Cref{thm: periodic geodesic in all homotopy classes of loops}.
By \Cref{thm:propertiesSobmetrics}, for $s=r$, the metric $G^L$ is a strong right-invariant metric on $\Diff^s(M)$, cf.~\cite{bauer2020well}. Condition~\ref{cond:non_vanishing_homotopy} holds by assumption. Moreover,
condition~\ref{cond:L_2 regular} has been checked implicitly in
\cite[Thm.~4.4, Thm.~5.8]{BruverisVialard2017}. 
It remains to check that \ref{cond:weaklyclosed} holds. For that let \(u_n\in A_{\id,[\alpha]}\) and assume that
\[
u_n\rightharpoonup u
\qquad\text{weakly in }
L^2([0,1],\mathfrak X^s(M)).
\]
Next we let
\(
\varphi_n:=\Evol(u_n)\) and 
\(\varphi:=\Evol(u)
\).
By \cite[Lemma~7.1]{BruverisVialard2017} applied with
\(\psi_0=\psi_1=\id\), the set
\[
\left\{
v\in L^2([0,1],\mathfrak X^s(M))
\ \middle|\
\evol(v)=\id
\right\}
\]
is weakly closed. Hence
\[
\operatorname{evol}(u)=\id.
\]
Thus \(\varphi\) is again a based loop at the identity.

It remains to show that \(\varphi\) represents the same based homotopy
class as the loops \(\varphi_n\). For that we note that for $s_0$ with $\frac{\on{dim}(M)}{2}+1<s_0<s$
it is shown in the proof of \cite[Lemma~7.1]{BruverisVialard2017}, that after passing to the weaker Sobolev topology:
\[
\varphi_n\longrightarrow \varphi
\qquad\text{in } C^0([0,1],\Diff^{s_0}(M)).
\]
Since \(\Diff^{s_0}(M)\) is a Hilbert manifold, sufficiently \(C^0\)-close
based loops are homotopic relative to their endpoints. Therefore, for
all sufficiently large \(n\),
\begin{equation}\label{eq:equality_induced_classes_in_Diff_s_0}
    [\varphi_n]=[\varphi]
\qquad\text{in } \pi_1(\Diff^{s_0}(M),\id).
\end{equation}
Let
\[
\iota:\Diff^{s}(M)\hookrightarrow \Diff^{s_0}(M)
\]
denote the natural inclusion. By the standard smoothing theorem for
Sobolev diffeomorphism groups, \(\iota\) is a weak homotopy equivalence;
in particular,
\[
\iota_*:\pi_1(\Diff^{s}(M),\id)\longrightarrow \pi_1(\Diff^{s_0}(M),\id)
\]
is injective. Since each \(\varphi_n\) represents \([\alpha]\) in
\(\pi_1(\Diff^{s}(M),\id)\), the equality in~\eqref{eq:equality_induced_classes_in_Diff_s_0} of the induced classes in
\(\pi_1(\Diff^{s_0}(M),\id)\) implies
\[
[\varphi]=[\alpha]
\qquad\text{in } \pi_1(\Diff^{s}(M),\id),
\]
from which we can conclude that
\[
u\in A_{\id,[\alpha]}.
\]
Hence \(A_{\id,[\alpha]}\) is weakly closed. Thus condition~\ref{cond:weaklyclosed} holds, and we can
apply \Cref{thm: periodic geodesic in all homotopy classes of loops} and thus also \Cref{cor:infinitely_many_periodic_geodeisc_in_each_homotopy_class}, which
finishes the proof.
\end{proof}

\subsection{Groups of Sobolev Diffeomorphisms Equipped with Right-Invariant Metrics Induced by Isometry Equivariant Inertia Operators}
Unfortunately the validity of the Palais–Smale-type hypothesis required in \Cref{thm:Lyusternik_half} is presently unknown in the setting of right-invariant metrics on \(\Diff^s(M)\). To still obtain the existence of 
contractible periodic geodesic, we will thus resort to \Cref{prop_dyn_reduction}. In contrast to the previous subsection, we will even obtain the
existence of periodic geodesics on \((\Diff^s(M),G^L)\) in the case where
\(\frac12\leq r<s\), i.e., in the regime where the induced Sobolev metric \(G^L\) is merely a weak Riemannian metric. The
price we pay to apply~\Cref{prop_dyn_reduction}  is that we have to impose suitable equivariance conditions on the
inertia operator \(L\) used in \eqref{eq:def_H_r_metric} to define \(G^L\).

First, we note that by \Cref{prop_dyn_reduction} it suffices to detect closed
finite-dimensional totally geodesic submanifolds of \(\Diff^s(M)\) in order to
obtain periodic geodesics on \((\Diff^s(M),G^L)\). Therefore, we begin by giving
a sufficient symmetry criterion for the inertia operator which guarantees that
the isometry group \(\Isom(M,g)\) is a totally geodesic submanifold of
\(\Diff^s(M)\). We denote the (identity component) of the group of isometries by
\[
    \cK:=\Isom_0(M,g)
\]
and recall that its Lie
algebra 
is precisely the space
\[
    \mathfrak k:=\mathrm{Lie}(\cK)=\Kill(M,g)
\]
of Killing vector fields on \((M,g)\). Next, we introduce the class of \(\cK\)-equivariant inertia operators:

\begin{definition}[\(\cK\)-equivariant inertia operators]
\label{def:K_equiv_inertia_operator}
Let \(L\) be a positive, self-adjoint, invertible, elliptic operator of order
\(2r\), with \(0\leq r\leq s\). We call \(L\) a \(\cK\)-equivariant inertia
operator of order \(2r\) if the following conditions hold:
\begin{enumerate}[label=(\alph*), ref=(\alph*)]
    \item \label{def:K_equiv_inertia_operator_preserves_killing}
    \(L\) preserves the space of Killing fields, i.e.
    \[
        L(\mathfrak k)\subseteq \mathfrak k .
    \]

    \item \label{def:K_equiv_inertia_operator_equivariant}
    \(L\) is \(\cK\)-equivariant, i.e.
    \[
        L(\rho_*u)=\rho_*L(u)
        \qquad
        \text{for all }\rho\in \cK
        \text{ and }u\in H^s(M,TM).
    \]
\end{enumerate}
We denote the space of such operators by \(\cI_\cK(2r)\).
\end{definition}
\begin{example} Returning to~\Cref{ex:laplace_inertia} of $L=(1-\Delta_g)^r$, we note that these inertia operators  
belong to \(\cI_\cK(2r)\) only under certain assumptions on the background metric $g$, e.g., that $g$ is Einstein. This is necessary to guarantee that $L$  preserves the space of Killing fields.  
\end{example}
The next lemma
shows that this example is far from isolated: the class \(\cI_\cK(2r)\) is
infinite-dimensional, although it has infinite codimension inside the full
space of inertia operators of order \(2r\).
\begin{lemma}[Size of the class \(\cI_\cK(2r)\)]
Let $(M,g)$ be a Riemannian manifold with isometry group $\cK$. Then the class \(\cI_\cK(2r)\) is nonempty and infinite-dimensional.
\end{lemma}
\begin{comment}
\begin{remark}[Codimension of the equivariant class]
Although Lemma~4.5 shows that \(I_K(2r)\) is infinite-dimensional, the equivariance
condition is still highly restrictive. If \(K\neq\{e\}\), then the corresponding linear space
of self-adjoint pseudodifferential operators satisfying
\[
L\rho_*=\rho_*L
\qquad
\text{for all } \rho\in K
\]
has infinite codimension inside the full space of self-adjoint pseudodifferential operators
of order at most \(2r\). This can already be seen at the level of principal symbols:
\(K\)-equivariance forces the principal symbol to be \(K\)-invariant. Since nontrivial
\(K\)-actions impose infinitely many independent linear constraints on smooth functions on
\(S^*M\), the resulting symbol space has infinite codimension. We leave the details to the
reader.
\end{remark}
\end{comment}
\begin{proof}
We start by showing that the space is nonempty. If $g$ is Einstein then the non-emptiness is provided by ~\Cref{ex:laplace_inertia}. In the general case, we 
let
\[
P_{\mathfrak{k}}:L^2(M,TM)\to \mathfrak{k}
\]
be the \(L^2\)-orthogonal projection onto the Lie algebra of $\cK$. Since \(\mathfrak{k}\) is finite-dimensional and consists
of smooth vector fields, \(P_{\mathfrak{k}}\) is a smoothing operator. Moreover,
\(P_{\mathfrak{k}}\) is \(\cK\)-equivariant, because \(\mathfrak{k}\) is \(\cK\)-invariant and the
\(\cK\)-action on \(L^2(M,TM)\) is unitary.

Next we consider again the inertia operator from~\Cref{ex:laplace_inertia}, i.e.,
\(
L=(1-\Delta_g)^r.
\)
Since \(K\) acts
by isometries, \(L\) is positive, self-adjoint, elliptic of order \(2r\), and \(K\)-equivariant. As noted above, in general $L$ does not  preserve the space
of Killing fields and in the following we will perturb the operator to obtain a new operator that satisfies this additional assumption. 
Define
\[
\tilde L
:=
P_{\mathfrak{k}}LP_{\mathfrak{k}}
+
(1-P_{\mathfrak{k}})L(1-P_{\mathfrak{k}}).
\]
Then \(\tilde L\) is self-adjoint and \(\cK\)-equivariant. Moreover, $\tilde L(\mathfrak{k})\subseteq \mathfrak{k}.$ Furthermore,
\[
\tilde L-L
=
-P_{\mathfrak{k}}L(1-P_{\mathfrak{k}})
-
(1-P_{\mathfrak{k}})LP_{\mathfrak{k}}.
\]
Both summands on the right-hand side are smoothing, since they contain the finite-rank
smoothing projection \(P_{\mathfrak{k}}\). Hence \(\tilde L\) and \(L\) have the same principal
symbol. Therefore \(\tilde L\) is elliptic of order \(2r\).

We next check positivity. Write
\[
u=v+w,
\qquad
v=P_{\mathfrak{k}}u\in\mathfrak{k},
\qquad
w=(1-P_{\mathfrak{k}})u\in\mathfrak{k}^{\perp}.
\]
Then
\[
G^{\tilde L}_{\id}(u,u)
=G^{L}_{\id}(v,v)
+
G^{L}_{\id}(w,w).
\]
Since \(L\) is positive and invertible, the right-hand side is strictly positive whenever
\(u\neq 0\). Thus \(\tilde L\) is positive. Since \(\tilde L\) is self-adjoint, elliptic, and has trivial
kernel, it is invertible. Hence
\(
\tilde L\in \cI_\cK(2r),
\)
which proves that \(\cI_\cK(2r)\) is nonempty.

It remains to show that \(\cI_\cK(2r)\) is infinite-dimensional. Let \(D\) be a positive
self-adjoint elliptic \(\cK\)-equivariant operator preserving \(\mathfrak{k}\), which always exists by the above argument. Since
\(D\) has compact resolvent, it has infinitely many finite-dimensional eigenspaces. Let
\(
Q_j:L^2(M,TM)\to E_j
\)
denote the corresponding spectral projections. Each \(Q_j\) is finite-rank, smoothing,
self-adjoint, \(\cK\)-equivariant, and preserves \(\mathfrak{k}\).

The operators \(Q_j\) are linearly independent. For every \(j\) and every \(t>0\), the operator
\(
\tilde L+tQ_j
\)
is still positive, self-adjoint, \(\cK\)-equivariant, and preserves \(\mathfrak{k}\). Moreover,
\(tQ_j\) is smoothing, so it does not change the principal symbol. Therefore \(\tilde L+tQ_j\)
is still elliptic of order \(2r\). Hence
\(
\tilde L+tQ_j\in I_K(2r)
\)
for every \(j\) and every \(t>0\). Since the directions \(Q_j\) are linearly independent,
\(I_K(2r)\) is infinite-dimensional.
\end{proof}

Having fixed this notation, we can now state a criterion ensuring that
\(\Isom(M,g)\) is a totally geodesic subgroup of \(\Diff^s(M)\).
\begin{lemma}[Criterion for the isometry group to be totally geodesic]
\label{prop:isometry_group_totally_geodesic}
Let \((M,g)\) be a closed Riemannian manifold, and let
\(L\in \cI_\cK(2r)\) as in \Cref{def:K_equiv_inertia_operator} with $\frac12\leq r\leq s$. Then
\(\Isom(M,g)\), equipped with the metric induced by restriction of \(G^L\), is a
totally geodesic submanifold of \((\Diff^s(M),G^L)\). %The result continues to hold when replacing $\Diff^s(M)$ with its subgroup of volume preserving diffeomorphisms $\Diff^s_{\mu}(M)$.
\end{lemma}

\begin{proof}
We first note that the condition $\frac12\leq r\leq s$ ensures that the geodesic equation exists and is locally well-posed, cf.~\Cref{thm:propertiesSobmetrics}. Next we show that the metric induced by \(G^L\) on \(\cK\) is bi-invariant.
By condition~\ref{def:K_equiv_inertia_operator_preserves_killing},
\(L(\mathfrak k)\subseteq \mathfrak k\). Hence \(L\) restricts to a positive
definite operator
\[
    L|_{\mathfrak k}\colon \mathfrak k\to \mathfrak k .
\]
For \(X,Y\in\mathfrak k\), the induced inner product at the identity is
\begin{equation}\label{eq:induced_metric_at_identity}
   G^L_{\mathrm{id}}(X,Y)
    =
    \int_M g(X,LY)\,\mathrm{dvol}_g .
\end{equation}

Let \(\rho\in K\). By
condition~\ref{def:K_equiv_inertia_operator_equivariant}, \(L\) commutes with
the push-forward action of \(\rho\). Since \(\rho\) also preserves \(g\) and
\(\mathrm{dvol}_g\), it follows that the restriction of
\(G^L_{\mathrm{id}}\) to \(\mathfrak k\) is \(\operatorname{Ad}(K)\)-invariant. Therefore the
induced metric on \(\cK\) is bi-invariant.

We next show that every one-parameter subgroup of \(\cK\) is a geodesic of
\((\Diff^s(M),G^L)\). Let \(X\in\mathfrak k\) and consider
\begin{equation}\label{eq:varphi_K_Diff}
        \varphi(t)=\exp(tX)\in K\subseteq \Diff^s(M).
\end{equation}
Its right logarithmic derivative is constant and equal to \(X\). Since
\(K\subseteq \Diff^s(M)\) consists of smooth diffeomorphisms, the transpose
adjoint operator \(\operatorname{ad}^{\top}\) is applicable to elements of \(\cK\)
by \cite[Lemma~7.4]{Bauer_2025}. Hence \(\varphi\) is a geodesic of
\((\Diff^s(M),G^L)\) if and only if
\(u=\partial_t\varphi\circ \varphi^{-1}\) solves the associated Euler--Arnold
equation
\[
    \partial_t u+\operatorname{ad}^{\top}_u u=0 .
\]
Since the right logarithmic derivative of \(\varphi\) is the constant vector
field \(u(t)=X\), this reduces to
\begin{equation}\label{eq:geodesic_eq_on_K}
        \operatorname{ad}^{\top}_X X=0 .
\end{equation}
It remains to prove \eqref{eq:geodesic_eq_on_K}.

Since the restriction of \(G^L_{\mathrm{id}}\) to \(\mathfrak k\) is
\(\on{Ad}(K)\)-invariant, it is \(\operatorname{ad}(\mathfrak k)\)-invariant.
Thus, for all \(X,Y,Z\in\mathfrak k\),
\[
    G^L_{\mathrm{id}}([X,Y],Z)
    +
    G^L_{\mathrm{id}}(Y,[X,Z])
    =
    0 .
\]
Taking \(Y=X\), and using \eqref{eq:induced_metric_at_identity} together with
the self-adjointness of \(L\), we obtain
\begin{equation}\label{eq:identity_G_r_mathfrakk}
     \int_M g(LX,[X,Z])\,\mathrm{dvol}_g=0
    \qquad
    \forall Z\in\mathfrak k .
\end{equation}

On the other hand, since \(X\) is a Killing field, its flow preserves \(g\) and
\(\mathrm{dvol}_g\). Hence, for every \(Y\in\mathfrak X(M)\),
\[
    0
    =
    \int_M \mathcal L_X\bigl(g(LX,Y)\,\mathrm{dvol}_g\bigr).
\]
Using \(\mathcal L_Xg=0\), \(\mathcal L_X\mathrm{dvol}_g=0\), and
\(LX\in\mathfrak k\), where the last inclusion follows from
condition~\ref{def:K_equiv_inertia_operator_preserves_killing}, this gives
\begin{equation}\label{eq:integration_by_parts_identity}
    \int_M g([X,LX],Y)\,\mathrm{dvol}_g
    +
    \int_M g(LX,[X,Y])\,\mathrm{dvol}_g
    =
    0
    \qquad
    \forall Y\in\mathfrak X(M).
\end{equation}
In particular, taking \(Y=Z\in\mathfrak k\), we get
\begin{equation}\label{eq:identity_left_right_on_z}
       \int_M g([X,LX],Z)\,\mathrm{dvol}_g
    +
    \int_M g(LX,[X,Z])\,\mathrm{dvol}_g
    =
    0
    \qquad
    \forall Z\in \mathfrak k .
\end{equation}
By \eqref{eq:identity_G_r_mathfrakk}, the second term in
\eqref{eq:identity_left_right_on_z} vanishes. Therefore
\[
    \int_M g([X,LX],Z)\,\mathrm{dvol}_g=0
    \qquad
    \forall Z\in\mathfrak k .
\]
Since \(X,LX\in\mathfrak k\) by
condition~\ref{def:K_equiv_inertia_operator_preserves_killing}, we have
\([X,LX]\in\mathfrak k\). Since the \(L^2\)-inner product is positive definite
on the finite-dimensional space \(\mathfrak k\), we conclude that
\begin{equation}\label{eq:comm_LX_X_zero}
        [X,LX]=0
        \qquad
        \forall X\in\mathfrak k .
\end{equation}

Recall that, by definition of \(\operatorname{ad}^{\top}\), we have
\[
    G^L_{\mathrm{id}}\bigl(\operatorname{ad}^{\top}_X X,Y\bigr)
    =
    G^L_{\mathrm{id}}\bigl(X,[X,Y]\bigr)
    \qquad
    \forall Y\in \mathfrak X(M).
\]
In combination with \eqref{eq:induced_metric_at_identity}, the self-adjointness
of \(L\), and \eqref{eq:integration_by_parts_identity}, this gives
\[
    G^L_{\mathrm{id}}\bigl(\operatorname{ad}^{\top}_X X,Y\bigr)
    =
    -\int_M g([X,LX],Y)\,\mathrm{dvol}_g
    \qquad
    \forall Y\in \mathfrak X(M).
\]
Together with \eqref{eq:comm_LX_X_zero} and the nondegeneracy of
\(G^L_{\mathrm{id}}\), this yields
\[
    \operatorname{ad}^{\top}_X X=0 .
\]
Thus, by \eqref{eq:geodesic_eq_on_K}, the one-parameter subgroup
\(\varphi\) in \eqref{eq:varphi_K_Diff} generated by \(X\) is a geodesic of
\((\Diff^s(M),G^L)\).

Since the induced metric on \(\cK\) is bi-invariant, the geodesics of \(\cK\) are
precisely the right translations of one-parameter subgroups. The ambient metric
\(G^L\) on \(\Diff^s(M)\) is right-invariant, so right translations preserve
ambient geodesics. It follows that every geodesic of \(\cK\) is also a geodesic of
\(\Diff^s(M)\). Hence \(\cK=\Isom_0(M,g)\) is totally geodesic.

Finally, if one works with the full isometry group \(\Isom(M,g)\), each
connected component is a right translate of \(\cK\). Since right translations are
isometries for \(G^L\), the same conclusion holds for \(\Isom(M,g)\).
\end{proof}
Finally, we are in a position to state the main result of this subsection, which  follows directly from~\Cref{prop_dyn_reduction} and ~\Cref{prop:isometry_group_totally_geodesic}:
\begin{corollary}[Infinitely many periodic geodesics on \(\Diff^s(M)\)]
\label{thm:periodic_geodesics_from_isometries}
Let \((M,g)\) be a closed, finite-dimensional Riemannian manifold, let
\(s>\frac{\on{dim}(M)}{2}+1\), and let \(L\in\cI_\cK(2r)\) be an inertia operator defining the
\(G^L\)-metric on \(\Diff^s(M)\) as in \eqref{eq:def_H_r_metric} with $\frac12\leq r\leq s$. Then the
following hold:
\begin{enumerate}
    \item If \(\pi_1(K)\neq 0\), then \((\Diff^s(M),G^L)\) admits a closed
    geodesic in each nontrivial homotopy class represented by a loop in \(\cK\).

    \item If there exists \(q\geq 2\) such that \(\pi_q(K)\neq 0\), then
    \((\Diff^s(M),G^L)\) admits a contractible periodic geodesic.
\end{enumerate}
In each case, there exist infinitely many geometrically distinct periodic
geodesics of the corresponding type.
%The result continues to hold when replacing $\Diff^s(M)$ with its subgroup of volume preserving diffeomorphisms $\Diff^s_{\mu}(M)$.
\end{corollary}
\begin{remark}[Periodic geodesics in the smooth diffeomorphism group]
In the setting of \Cref{thm:periodic_geodesics_from_isometries}, the periodic
geodesics obtained from \(K=\Isom_0(M,g)\) are in fact smooth. Indeed, if
\(\gamma\) is a closed geodesic of \(\cK\), equipped with the metric induced by
\(G^L\), then \(\gamma\) is also a geodesic of \((\Diff^s(M),G^L)\) by
\Cref{prop:isometry_group_totally_geodesic}. Since \(\cK\) is a finite-dimensional
Lie group of smooth diffeomorphisms, \(\gamma\) lies entirely in
\(\Diff^\infty(M)\subseteq \Diff^s(M)\).

Moreover, by right-invariance of \(G^L\), for every
\(\varphi\in \Diff^\infty(M)\) the right translate
\[
    t\mapsto \gamma(t)\circ \varphi
\]
is again a periodic geodesic of \((\Diff^s(M),G^L)\), and it also lies entirely
in \(\Diff^\infty(M)\).

This smoothness statement is special to the present construction. In contrast,
the existence result in \Cref{cor:infinitely many periodic geodesics sobolev diffeos}
is obtained from \Cref{thm: periodic geodesic in all homotopy classes of loops}
by a variational argument which uses the Hilbert manifold structure of
\(\Diff^s(M)\) and the strongness of the \(G^L\)-metric. That argument produces
periodic geodesics in the Sobolev diffeomorphism group, but does not by itself
imply that they lie in the smooth subgroup \(\Diff^\infty(M)\).
\end{remark}

\begin{remark}[Examples for the homotopy conditions on \(K=\Isom_0(M,g)\)]
\label{rem:some_pi_q_nopn_zero_of_Isom}
The conditions \(\pi_1(K)\neq 0\) and \(\pi_q(K)\neq 0\) for some \(q\geq 2\)
hold for the following non-exclusive list of Riemannian manifolds:
\begin{enumerate}
    \item for \(M=S^1\) with its standard metric. Then
    \[
        K=\Isom_0(S^1,g_{\mathrm{std}})\cong SO(2)\cong S^1,
    \]
    and hence \(\pi_1(K)\cong\mathbb Z\);

    \item for \(M=S^n\), \(n\geq 2\), with the round metric. Then
    \[
        K=\Isom_0(S^n,g_{\mathrm{round}})\cong SO(n+1),
    \]
    so \(\pi_1(K)\cong\mathbb Z/2\) and also
    \[
        \pi_3(K)\neq 0;
    \]

    \item for \(M=T^n\) with a flat metric. Then
    \[
        K=\Isom_0(T^n,g_{\mathrm{flat}})\cong T^n,
    \]
    and hence \(\pi_1(K)\cong\mathbb Z^n\);

    \item for total spaces \(M\) of principal \(S^1\)-bundles
    \[
        S^1\hookrightarrow M\to B
    \]
    equipped with an \(S^1\)-invariant connection metric, provided that the
    fiber represents a nontrivial element of \(\pi_1(M)\). The fiberwise
    \(S^1\)-action is then by isometries, and evaluation at a point detects the
    fiber class. Hence \(\pi_1(K)\neq0\). For instance, the lens spaces
    \[
        S^1\hookrightarrow L(k,1)\to S^2,\qquad k\geq 2,
    \]
    with a connection metric give such examples;

    \item for \(M=\mathbb{CP}^m\), \(m\geq 1\), with the Fubini--Study metric.
    Then
    \[
        K=\Isom_0(\mathbb{CP}^m,g_{\mathrm{FS}})\cong PU(m+1),
    \]
    and hence
    \[
        \pi_1(K)\cong \mathbb Z/(m+1)\mathbb Z,
        \qquad
        \pi_3(K)\cong\mathbb Z;
    \]

    \item more generally, for compact homogeneous spaces \(M=G/H\) with a
    \(G\)-invariant metric, whenever the induced action of \(G\) on \(M\) gives
    a nontrivial class in some \(\pi_q(\Isom_0(M,g))\), one obtains
    \[
        \pi_q(K)\neq0.
    \]
    In particular, compact symmetric spaces with non-abelian transvection group
    give many examples with \(\pi_3(K)\neq0\).
\end{enumerate}
\end{remark}

\subsection{A Synthetic Example for the Lyusternik–-Fet Theorem}
\label{sec:ex:synthetic}
We conclude this section with a synthetic example for which the Palais--Smale
condition modulo right translations can be verified explicitly, which in turn allows us to apply~
\Cref{thm:Lyusternik_half}.

Therefore let \(\cH\) be an infinite-dimensional separable Hilbert space and let
\(\cK\) be a compact non-aspherical Lie group. We consider the Hilbert Lie group
\[
        \cG := \cH \times \cK
\]
with product group structure \((h,A)(k,B) := (h+k,AB)\) and Lie algebra
\[
        \mathfrak g = \cH\oplus \mathfrak{k}.
\]
Next we choose a non-zero linear map
\begin{equation}\label{def:synth_example_of_C}
            \chi:\mathfrak{k}\longrightarrow \cH
\end{equation}
and a constant \(\lambda>0\). We define an inner product on \(\mathfrak g\) by
\[
        \big\langle (h,X),(k,Y)\big\rangle_{\mathfrak g}
        :=
        \langle h+\chi(X),k+\chi(Y)\rangle_{\cH}
        +
        \lambda\langle X,Y\rangle_{\mathfrak{k}},
\]
where \(\langle\cdot,\cdot\rangle_{\mathfrak{k}}\) is a fixed
bi-invariant inner product on \(\mathfrak{k}\). Extending this inner
product by right translations defines a smooth strong right-invariant
Riemannian metric on \(\cG\).  For a curve
\[
        \gamma(t)=(x(t),A(t))
\]
the corresponding energy functional is 
\begin{equation}\label{eq:synthetic}
        E(x,A)
        =
        \frac12\int_0^1
        \left(
        \|\dot x+\chi(u)\|_{\cH}^2
        +
        \lambda\|u\|_{\mathfrak{k}}^2
        \right)\,\mathrm{d} t ,
\end{equation}
where
\[
        u(t):=\dot A(t)A(t)^{-1}\in\mathfrak{k}
\]
is the right logarithmic derivative of the \(\cK\)-component. Next, we show that the energy functional satisfies the Palais--Smale condition:
%We first collect a useful result on energy functionals on finite-dimensional manifolds:
% \begin{lemma}
% \label{lem:synthetic-fd-ps}
% Let \(M\) be a compact finite-dimensional manifold and let
% \(L:TM\to\mathbb R\) be a smooth Lagrangian which is uniformly positive
% definite and quadratic in the velocity variable. Then the corresponding energy
% functional on \(H^1\)-loops in \(M\) satisfies the Palais--Smale condition.
% \end{lemma}

% \begin{proof}
% This is the standard Palais--Smale compactness theorem for smooth uniformly
% positive definite Lagrangians on compact finite-dimensional manifolds.
% Bounded action gives an \(H^1\)-bound. Compactness of \(M\) gives, after
% passing to a subsequence, uniform convergence and weak \(H^1\)-convergence.
% The Palais--Smale condition excludes loss of kinetic energy and hence upgrades
% weak \(H^1\)-convergence to strong \(H^1\)-convergence.
% \end{proof}

\begin{proposition}
\label{prop:synthetic-ps}
The energy functional \(E\), as defined in~\eqref{eq:synthetic}, satisfies the Palais--Smale condition modulo right
translations on \(\Lambda_0\cG\).
\end{proposition}

\begin{proof}
Since \(\cG\) is a Hilbert Lie group, right translations identify the quotient
\(\Lambda_0\cG/\cG\) with the based loop space \(\Omega_e\cG\). Therefore it is
enough to prove the Palais--Smale condition on
\[
        \Omega_e\cG
        =
        \Omega_0\cH\times\Omega_e\cK,
\]
where
\[
        \Omega_0\cH
        :=
        \{x\in H^1([0,1],\cH):x(0)=x(1)=0\}.
\]
Therefore, let
\[
        \gamma_n=(x_n,A_n)
\]
be a Palais--Smale sequence for \(E\) on \(\Omega_e\cG\). Define
\begin{equation}\label{eq:def_u_n_p_n_in_synth_example}
       u_n:=\dot A_nA_n^{-1},
        \qquad
        p_n:=\dot x_n+\chi(u_n) .
\end{equation}
Since \(E(\gamma_n)\) is bounded, the sequences \((p_n)\) and \((u_n)\) are
bounded in \(L^2([0,1],\cH)\) and \(L^2([0,1],\mathfrak{k})\), respectively.
In particular, \((A_n)\) is bounded in \(H^1([0,1],\cK)\).

We first use variations in the \(\cH\)-direction. For every
\(\xi\in\Omega_0\cH\), we have
\begin{equation}\label{eq:express_de_in_x}
    d_xE(x_n,A_n)(\xi)
        =
        \int_0^1
        \langle p_n,\dot\xi\rangle_{\cH}\,\mathrm{d} t .
\end{equation}
Since \((\gamma_n)\) is a Palais--Smale sequence, this expression tends to zero
uniformly for \(\xi\) bounded in \(H^1\), that is
\begin{equation}\label{eq:uniformly_to_zeor_dE}
  \varepsilon_n:=  \sup_{\xi\in \Omega_0\cH\,,\,  \|\xi\|_{H^1}\le 1}
    \left|
        \int_0^1 \langle p_n(t),\dot\xi(t)\rangle_{\cH}\,\mathrm{d} t
    \right|
    \longrightarrow 0 .
\end{equation}
 Next we note that the derivative map
\[
        \Omega_0\cH\longrightarrow L^2([0,1],\cH),
        \qquad
        \xi\longmapsto \dot\xi,
\]
identifies \(\Omega_0\cH\) with the closed subspace
\[
        L^2_0([0,1],\cH)
        :=
        \left\{
        v\in L^2([0,1],\cH):
        \int_0^1 v(t)\,\mathrm{d} t=0
        \right\}.
\]
Here we denote by \(\Pi_0:L^2([0,1],\cH)\to L^2_0([0,1],\cH)\) the orthogonal projection onto this subspace.
Next we choose \(\xi_n\in\Omega_0\cH\) such that \(\dot\xi_n=\Pi_0(p_n)\). By Poincaré's
inequality it holds that
\[
        \|\xi_n\|_{H^1}
        \le C\|\Pi_0(p_n)\|_{L^2},
\]
for some constant $C>0$.
In combination with~\eqref{eq:uniformly_to_zeor_dE}, the identity
\(\dot\xi_n=\Pi_0p_n\), and the definition of \(\Pi_0\), this implies that
\[
        \|\Pi_0(p_n)\|_{L^2}^2
        \le
        C\varepsilon_n\|\Pi_0(p_n)\|_{L^2}.
\]
Hence by using again \eqref{eq:uniformly_to_zeor_dE} we have \[
\Pi_0(p_n)\to0 \quad \text{strongly in } L^2([0,1],\cH).
\]
By setting $\bar p_n:=\int_0^1 p_n(t)\,\mathrm{d} t$ for the mean of $p_n$ this convergence statement is equivalent to 
\begin{equation}\label{eq:convergence of p_n_to_bar_p}
    p_n-\bar p_n\to0
        \qquad\text{strongly in }L^2([0,1],\cH).
\end{equation}
Since \(x_n\) is a based loop, it holds that \(x_n(0)=x_n(1)=0\); thus we have
\[
        \int_0^1 \dot x_n(t)\,\mathrm{d}t = 0.
\]
From this we can conclude, by the definitions of \(u_n\) and \(p_n\) in
\eqref{eq:def_u_n_p_n_in_synth_example} and the fact that
\(\chi\) (see \eqref{def:synth_example_of_C}), as a bounded linear operator, commutes with the Bochner integral, that
\[
        \bar p_n
        =
        \chi\left(\int_0^1 u_n(t)\,\mathrm{d}t \right).
\]
By~\eqref{eq:def_u_n_p_n_in_synth_example} the right-hand side is an element of the  finite-dimensional subspace
\(\chi(\mathfrak{k})\subset \cH\), where $\chi$ is as in~\eqref{def:synth_example_of_C}. Thus $(\bar p_n)$ is a bounded sequence in the finite-dimensional subspace
\(\chi(\mathfrak{k})\subset \cH\). Hence, after passing to a subsequence, we may
therefore assume that
\[
        \bar p_n\longrightarrow \bar p
        \quad\text{strongly in }\cH
\]
for some \(\bar p\in \chi(\mathfrak{k})\). Consequently if we regard $\bar p_n, \bar p$ as constant $\cH$-valued functions on $[0,1]$ we have
\begin{equation}\label{eq:conv_bar_p_n_to_bar_p}
    \bar p_n\longrightarrow \bar p
        \quad\text{strongly in }L^2([0,1],\cH).
\end{equation}
We now consider variations in the \(\cK\)-direction. Let \(\eta\) be a
variation field along \(A_n\), and write \(v:=\eta A_n^{-1}.\)
Then the variation of the right logarithmic derivative is by standard computations (see for example~\cite{MardsenRatiuMechanics})
\[
        \delta u_n=\dot v+[v,u_n].
\]
The \(\cK\)-part of the first variation of $E$ on $\Omega_e\cG$ is therefore
\[
        d_AE(x_n,A_n)[\eta]
        =
        \int_0^1
        \left(
        \lambda\langle u_n,\delta u_n\rangle_{\mathfrak{k}}
        +
        \langle p_n,\chi(\delta u_n)\rangle_{\cH}
        \right)\,\mathrm{d} t .
\]
We next introduce an auxiliary functional on \(\Omega_e\cK\) whose first
variation matches the \(\cK\)-part of \(dE(x_n,A_n)\) up to an \(o(1)\)-term.
Therefore we define
\[
        F_{\bar p}:\Omega_e\cK\to\mathbb R,
        \qquad
        F_{\bar p}(A)
        :=
        \int_0^1
        \left(
        \frac{\lambda}{2}\|u\|_{\mathfrak{k}}^2
        +
        \langle \bar p,\chi(u)\rangle_\cH
        \right)\,\mathrm{d} t,
        \qquad
        u=\dot AA^{-1}.
\]
For \(\eta\in T_{A_n}\Omega_e\cK\), write \(v=\eta A_n^{-1}\). Then
\(\delta u_n=\dot v+[v,u_n]\), and
\[
        dF_{\bar p}(A_n)[\eta]
        =
        \int_0^1
        \left(
        \lambda\langle u_n,\delta u_n\rangle_{\mathfrak{k}}
        +
        \langle \bar p,\chi(\delta u_n)\rangle_\cH
        \right)\,\mathrm{d} t .
\]
Thus
\[
        d_AE(x_n,A_n)[\eta]-dF_{\bar p}(A_n)[\eta]
        =
        \int_0^1
        \langle p_n-\bar p,\chi(\delta u_n)\rangle_\cH\,\mathrm{d} t .
\]
Since \(\chi\) is bounded, \((u_n)\) is bounded in \(L^2\), and
\(H^1([0,1])\hookrightarrow L^\infty([0,1])\), we have
\[
        \|\delta u_n\|_{L^2}\le C\|\eta\|_{H^1}.
\]
Hence by \eqref{eq:convergence of p_n_to_bar_p} we obtain 
\[
        \left|
        d_AE(x_n,A_n)[\eta]-dF_{\bar p}(A_n)[\eta]
        \right|
        \le
        C\|p_n-\bar p\|_{L^2}\|\eta\|_{H^1}
        =
        o(1)\|\eta\|_{H^1}.
\]
Therefore \(dF_{\bar p}(A_n)\to0\). Moreover, \(F_{\bar p}(A_n)\) is bounded
because \((u_n)\) is bounded in \(L^2\); after passing to a subsequence, it
converges. Hence \((A_n)\) is a Palais--Smale sequence for \(F_{\bar p}\) on
\(\Omega_e\cK\).
%Indeed, the difference is
%\[
 %       \int_0^1
  %      \langle p_n-\bar p,C\delta u_n\rangle_H\,dt,
%\]
%and \(\delta u_n=\dot v+[v,u_n]\) is bounded in
%\(L^2([0,1],\mathfrak{k})\) whenever \(\eta\) is bounded in \(H^1\), since
%\((u_n)\) is bounded in \(L^2\) and \(\cK\) is finite-dimensional.
%Thus \((A_n)\) is a Palais--Smale sequence for \(F_{\bar p}\) on \(\Omega_e\cK\).

Since \(F_{\bar p}\) is the action functional of a smooth fiberwise quadratic
Lagrangian with uniformly positive definite quadratic part on the compact
finite-dimensional manifold \(\cK\), the classical Palais--Smale theorem for
Tonelli Lagrangians on compact finite-dimensional manifolds applies; see
e.g.~\cite{klingenberg1995riemannian}. Hence \(F_{\bar p}\) satisfies the
Palais--Smale condition on \(\Omega_e\cK\). Thus, after passing to a further
subsequence,
\[
        A_n\longrightarrow A
        \quad\text{strongly in }H^1([0,1],\cK).
\]
In particular, $u_n\to u:=\dot AA^{-1}$ strongly in \(L^2([0,1],\mathfrak{k})\). Since \(\chi:\mathfrak{k}\to \cH\) as in \eqref{def:synth_example_of_C} is linear, and $\mathfrak k$ is finite-dimensional, $\chi$ is hence bounded. This implies
\[
        \chi u_n\longrightarrow \chi u
        \quad\text{strongly in }L^2([0,1],\cH).
\]
Combining this with the definition of \(p_n\) in
\eqref{eq:def_u_n_p_n_in_synth_example} and the convergence statements
\eqref{eq:conv_bar_p_n_to_bar_p} and
\eqref{eq:convergence of p_n_to_bar_p}, we obtain
\[
        \dot x_n=p_n-\chi u_n
        \longrightarrow \bar p-\chi u
        \quad\text{strongly in }L^2([0,1],\cH).
\]
Since each \(x_n\) is a based loop, in particular \(x_n(0)=0\). Hence
\((x_n)\) converges strongly in \(H^1([0,1],\cH)\). Since also
\(A_n\to A\) strongly in \(H^1([0,1],\cK)\), it follows that
\((x_n,A_n)\) has a strongly convergent subsequence in \(\Omega_e\cG\).
This proves that the Palais--Smale condition modulo right translations holds
for the energy functional on \(\Lambda_0\cG\).
\end{proof}

\begin{corollary}
\label{cor:synthetic-Lyusternik}
The Hilbert half-Lie group \((\cG,G)\) satisfies the assumptions of
\Cref{thm:Lyusternik_half}. In particular, it admits a nonconstant contractible
periodic geodesic.
\end{corollary}

\begin{proof}
The metric constructed above is smooth, strong, and right-invariant. Moreover,
\(\H\) is contractible, and hence for all $p\geq 2$
\[
        \pi_p(\cG)
        \cong
        \pi_p(\cH)\oplus\pi_p(\cK)
        \cong
        \pi_p(\cK).
\]
As $\cK$ is not aspherical, there exists a $p\geq 3$ so that $\pi_p(\cK)\neq 0$, thus by the equation above \(G\) is not aspherical too. By \Cref{prop:synthetic-ps}, the energy
functional satisfies the Palais--Smale condition modulo right translations on
\(\Lambda_0\cG\). Hence all assumptions of \Cref{thm:Lyusternik_half} are
satisfied.
\end{proof}

\begin{remark}
For a generic choice of \(\chi\), the compact subgroup
\[
        \{0\}\times \cK\subset \cH\times \cK
\]
is not totally geodesic. Indeed, the coupling term \(\chi(u)\) in the energy couples
the \(\cH\)- and \(\cK\)-directions. More explicitly, the restriction of the
inner product to \(\mathfrak{k}\) contains the term
\[
        \langle \chi(X),\chi(Y)\rangle_\cH ,
\]
which is not \(\operatorname{Ad}\)-invariant for a generic choice of \(\chi\).
Consequently, the Euler--Arnold equation does not preserve the subspace
\(\{0\}\oplus\mathfrak{k}\). Thus geodesics initially tangent to the
\(\cK\)-factor need not remain in that factor, and the periodic geodesic
obtained from \Cref{thm:Lyusternik_half} is not produced by the elementary
finite-dimensional reduction principle of \Cref{prop_dyn_reduction}.
\end{remark}

\bibliographystyle{abbrv}
	\bibliography{ref_new}

@misc{bauer2025completeness,
  author        = {Bauer, Martin and Maor, Cy and Wirth, Benedikt},
  title         = {Completeness of reparametrization-invariant {Sobolev} metrics on the space of surfaces},
  year          = {2025},
  archivePrefix = {arXiv},
  eprint        = {2512.01566}
}

@misc{maor2022riemannian,
  author = {Maor, Cy},
  title  = {Riemannian geometry of diffeomorphism groups},
  note   = {Lecture notes from a course held at the Hebrew University},
  year   = {2022}
}

@book{schmeding2022introduction,
  title={An introduction to infinite-dimensional differential geometry},
  author={Schmeding, Alexander},
  volume={202},
  year={2022},
  publisher={Cambridge University Press}
}

@article{bauer2015local,
  author  = {Bauer, Martin and Escher, Joachim and Kolev, Boris},
  title   = {Local and global well-posedness of the fractional order {EPDiff} equation on {$\mathbb{R}^d$}},
  journal = {Journal of Differential Equations},
  volume  = {258},
  number  = {6},
  pages   = {2010--2053},
  year    = {2015},
  doi     = {10.1016/j.jde.2014.11.021}
}

@article{bauer2020well,
  author  = {Bauer, Martin and Bruveris, Martins and Cismas, Emanuel and Escher, Joachim and Kolev, Boris},
  title   = {Well-posedness of the {EPDiff} equation with a pseudo-differential inertia operator},
  journal = {Journal of Differential Equations},
  volume  = {269},
  number  = {1},
  pages   = {288--325},
  year    = {2020},
  doi     = {10.1016/j.jde.2019.12.008}
}

@article{EM70,
    AUTHOR = {Ebin, David G. and Marsden, Jerrold},
     TITLE = {Groups of diffeomorphisms and the motion of an incompressible
              fluid},
   JOURNAL = {Ann. of Math. (2)},
  FJOURNAL = {Annals of Mathematics. Second Series},
    VOLUME = {92},
      YEAR = {1970},
     PAGES = {102--163},
      ISSN = {0003-486X},
   MRCLASS = {57.47 (76.00)},
  MRNUMBER = {271984},
       DOI = {10.2307/1970699},
       URL = {https://doi.org/10.2307/1970699},
}

@article{Str90,
  author    = {M. Struwe},
  title     = {{Existence of periodic solutions of Hamiltonian systems on almost every energy surface}},
  journal   = {Boletim da Sociedade Brasileira de Matemática},
  volume    = {20},
  year      = {1990},
  pages     = {49--58}
}

@book{klingenberg1995riemannian,
  author    = {Klingenberg, Wilhelm P. A.},
  title     = {Riemannian Geometry},
  series    = {De Gruyter Studies in Mathematics},
  volume    = {1},
  edition   = {2nd revised ed.},
  publisher = {Walter de Gruyter},
  address   = {Berlin},
  year      = {1995},
  doi       = {10.1515/9783110905120}
}

@article {Abbo13Lect,
	AUTHOR = {Abbondandolo, Alberto},
	TITLE = {Lectures on the free period {L}agrangian action functional},
	JOURNAL = {J. Fixed Point Theory Appl.},
	FJOURNAL = {Journal of Fixed Point Theory and Applications},
	VOLUME = {13},
	YEAR = {2013},
	NUMBER = {2},
	PAGES = {397--430},
	ISSN = {1661-7738,1661-7746},
	MRCLASS = {37J45},
	MRNUMBER = {3122334},
	DOI = {10.1007/s11784-013-0128-1},
	URL = {https://doi.org/10.1007/s11784-013-0128-1},
}

@misc{HopfRinowHalfLiegroups,
  author        = {Maier, Levin and Ruscelli, Francesco},
  title         = {The {Hopf--Rinow} Theorem and {Ma{\~n}{\'e}}'s Critical Value for Magnetic Geodesics on Half {Lie}-Groups},
  year          = {2025},
  archivePrefix = {arXiv},
  eprint        = {2510.19323}
}

@misc{TonelliHalfLiegroups,
  author        = {Maier, Levin and Ruscelli, Francesco},
  title         = {{On {Ma{\~n}{\'e}}'s Critical Value for Tonelli Lagrangians on Half {Lie}-Groups}},
  year          = {2025},
  archivePrefix = {arXiv},
  eprint        = {2511.13428}
}

@book{La99,
	author = {Lang, Serge},
	doi = {10.1007/978-1-4612-0541-8},
	isbn = {0-387-98593-X},
	mrclass = {53-01 (58-01)},
	mrnumber = {1666820},
	mrreviewer = {Man\ Chun\ Leung},
	pages = {xviii+535},
	publisher = {Springer-Verlag, New York},
	series = {Graduate Texts in Mathematics},
	title = {Fundamentals of differential geometry},
	url = {https://doi.org/10.1007/978-1-4612-0541-8},
	volume = {191},
	year = {1999}}

@article{misiolek2010fredholm,
  author  = {Misio{\l}ek, Gerard and Preston, Stephen C.},
  title   = {Fredholm properties of {Riemannian} exponential maps on diffeomorphism groups},
  journal = {Inventiones Mathematicae},
  volume  = {179},
  number  = {1},
  pages   = {191--227},
  year    = {2010},
  doi     = {10.1007/s00222-009-0217-3}
}

@article{escher2014right,
  author  = {Escher, Joachim and Kolev, Boris},
  title   = {Right-invariant {Sobolev} metrics of fractional order on the diffeomorphism group of the circle},
  journal = {Journal of Geometric Mechanics},
  volume  = {6},
  number  = {3},
  pages   = {335--372},
  year    = {2014},
  doi     = {10.3934/jgm.2014.6.335}
}

@article{bauer2014homogeneous,
  author  = {Bauer, Martin and Bruveris, Martins and Michor, Peter W.},
  title   = {The homogeneous {Sobolev} metric of order one on diffeomorphism groups on the real line},
  journal = {Journal of Nonlinear Science},
  volume  = {24},
  number  = {5},
  pages   = {769--808},
  year    = {2014},
  doi     = {10.1007/s00332-014-9204-y}
}

@article{marquis2018half,
  author  = {Marquis, Timoth{\'e}e and Neeb, Karl-Hermann},
  title   = {Half-{Lie} groups},
  journal = {Transformation Groups},
  volume  = {23},
  number  = {3},
  pages   = {801--840},
  year    = {2018},
  doi     = {10.1007/s00031-018-9485-6}
}

@book {HoferZehnderBook,
    AUTHOR = {Hofer, Helmut and Zehnder, Eduard},
     TITLE = {Symplectic invariants and {H}amiltonian dynamics},
    SERIES = {Birkh\"auser Advanced Texts: Basler Lehrb\"ucher.
              [Birkh\"auser Advanced Texts: Basel Textbooks]},
 PUBLISHER = {Birkh\"auser Verlag, Basel},
      YEAR = {1994},
     PAGES = {xiv+341},
      ISBN = {3-7643-5066-0},
   MRCLASS = {58-02 (34C25 57R15 58E05 58F05 70H05)},
  MRNUMBER = {1306732},
MRREVIEWER = {Daniel\ M.\ Burns, Jr.},
       DOI = {10.1007/978-3-0348-8540-9},
       URL = {https://doi.org/10.1007/978-3-0348-8540-9},
}

@article{HoferZehnder1987,
  author    = {H. Hofer and E. Zehnder},
  title     = {{Periodic solutions on hypersurfaces and a result by C. Viterbo}},
  journal   = {Inventiones Mathematicae},
  volume    = {90},
  pages     = {1--9},
  year      = {1987},
  doi       = {10.1007/BF01389030},
  publisher = {Springer},
}

@article {LystFetThm51,
    AUTHOR = {Lyusternik, L. A. and Fet, A. I.},
     TITLE = {Variational problems on closed manifolds},
   JOURNAL = {Doklady Akad. Nauk SSSR (N.S.)},
  FJOURNAL = {Doklady Akad. Nauk SSSR (N.S.)},
    VOLUME = {81},
      YEAR = {1951},
     PAGES = {17--18},
   MRCLASS = {49.0X},
  MRNUMBER = {44760},
MRREVIEWER = {L.\ C.\ Young},
}

@article{Arnold66,
     author = {Arnold, V. I.michor},
     title = {Sur la g\'eom\'etrie diff\'erentielle des groupes de {Lie} de dimension infinie et ses applications \`a l'hydrodynamique des fluides parfaits},
     journal = {Annales de l'Institut Fourier},
     pages = {319--361},
     publisher = {Institut Fourier},
     address = {Grenoble},
     volume = {16},
     number = {1},
     year = {1966}
}

@article{Bauer_2025,
  author  = {Bauer, Martin and Harms, Philipp and Michor, Peter W.},
  title   = {Regularity and completeness of half-{Lie} groups},
  journal = {Journal of the European Mathematical Society},
  year    = {2025},
  note    = {Published online first},
  doi     = {10.4171/JEMS/1587}
}

@article{Vi08,
    AUTHOR = {Vizman, Cornelia},
     TITLE = {Geodesic equations on diffeomorphism groups},
   JOURNAL = {SIGMA Symmetry Integrability Geom. Methods Appl.},
  FJOURNAL = {SIGMA. Symmetry, Integrability and Geometry. Methods and
              Applications},
    VOLUME = {4},
      YEAR = {2008},
     PAGES = {Paper 030, 22},
      ISSN = {1815-0659},
   MRCLASS = {37K65 (35A30 35Q35 58D05)},
  MRNUMBER = {2393297},
MRREVIEWER = {Stephen\ Carl\ Preston},
       DOI = {10.3842/SIGMA.2008.030},
       URL = {https://doi.org/10.3842/SIGMA.2008.030},
}

@book {MardsenRatiuMechanics,
    AUTHOR = {Marsden, Jerrold E. and Ratiu, Tudor S.},
     TITLE = {Introduction to mechanics and symmetry},
    SERIES = {Texts in Applied Mathematics},
    VOLUME = {17},
   EDITION = {Second},
      NOTE = {A basic exposition of classical mechanical systems},
 PUBLISHER = {Springer-Verlag, New York},
      YEAR = {1999},
     PAGES = {xviii+582},
      ISBN = {0-387-98643-X},
   MRCLASS = {70-02 (37Jxx 53Dxx 70Gxx 70Hxx)},
  MRNUMBER = {1723696},
       DOI = {10.1007/978-0-387-21792-5},
       URL = {https://doi.org/10.1007/978-0-387-21792-5},
}

@article{EliashbergPolterovich1993,
  author       = {Eliashberg, Y. and Polterovich, L.},
  title        = {{Bi-invariant Metrics on the Group of Hamiltonian Diffeomorphisms}},
  journal      = {International Journal of Mathematics},
  volume       = {4},
  pages        = {727--738},
  year         = {1993},
  zblnumber    = {0795.58016},
  mrnumber     = {1245350}
}

@article{BauerHarmsPreston2020,
  author  = {Bauer, Martin and Harms, Philipp and Preston, Stephen C.},
  title   = {Vanishing distance phenomena and the geometric approach to {SQG}},
  journal = {Archive for Rational Mechanics and Analysis},
  volume  = {235},
  pages   = {1445--1466},
  year    = {2020},
  doi     = {10.1007/s00205-019-01449-7}
}

@article{BruverisVialard2017,
  author  = {Bruveris, Martins and Vialard, Fran{\c c}ois-Xavier},
  title   = {On completeness of groups of diffeomorphisms},
  journal = {Journal of the European Mathematical Society},
  volume  = {19},
  number  = {5},
  pages   = {1507--1544},
  year    = {2017},
  doi     = {10.4171/JEMS/698}
}

@book{Palais68,
  author       = {Palais, R. S.},
  title        = {{Foundations of Global Non-Linear Analysis}},
  publisher    = {W. A. Benjamin, Inc.},
  address      = {New York--Amsterdam},
  year         = {1968},
  zblnumber    = {0164.11102},
  mrnumber     = {0248880}
}

@article{Atkin97,
  author       = {Atkin, C. J.},
  title        = {Geodesic and metric completeness in infinite dimensions},
  journal      = {Hokkaido Mathematical Journal},
  volume       = {26},
  pages        = {1--61},
  year         = {1997},
  zblnumber    = {0871.58006},
  mrnumber     = {1432537}
}

@book{AK98,
    AUTHOR = {Arnold, Vladimir I. and Khesin, Boris A.},
     TITLE = {Topological methods in hydrodynamics},
    SERIES = {Applied Mathematical Sciences},
    VOLUME = {125},
 PUBLISHER = {Springer-Verlag, New York},
      YEAR = {1998},
     PAGES = {xvi+374},
      ISBN = {0-387-94947-X},
   MRCLASS = {58-02 (35Q30 58B25 58D05 76-02 76M30)},
  MRNUMBER = {1612569},
MRREVIEWER = {Nikolai\ K.\ Smolentsev},
}

@article{kriegl2015exotic,
  title={{An exotic zoo of diffeomorphism groups on $\bR^n$}},
  author={Kriegl, Andreas and Michor, Peter W and Rainer, Armin},
  journal={Annals of Global Analysis and Geometry},
  volume={47},
  number={2},
  pages={179--222},
  year={2015},
  publisher={Springer}
}

@article{BauerBruverisMichor2014Overview,
  author  = {Bauer, Martin and Bruveris, Martins and Michor, Peter W.},
  title   = {Overview of the geometries of shape spaces and diffeomorphism groups},
  journal = {Journal of Mathematical Imaging and Vision},
  volume  = {50},
  number  = {1--2},
  pages   = {60--97},
  year    = {2014},
  doi     = {10.1007/s10851-013-0490-z}
}

@book{SrivastavaKlassen2016,
  author    = {Srivastava, Anuj and Klassen, Eric P.},
  title     = {{Functional and Shape Data Analysis}},
  series    = {{Springer Series in Statistics}},
  publisher = {{Springer-Verlag}},
  address   = {{New York}},
  year      = {2016},
  doi       = {10.1007/978-1-4939-4020-2},
  mrnumber  = {3821566},
  zbl       = {1376.62003}
}

@article{BenamouBrenier2000,
  author   = {Benamou, Jean-David and Brenier, Yann},
  title    = {{A computational fluid mechanics solution to the {Monge}-{Kantorovich} mass transfer problem}},
  journal  = {{Numerische Mathematik}},
  volume   = {84},
  number   = {3},
  year     = {2000},
  pages    = {375--393},
  doi      = {10.1007/s002110050002}
}

@article{Otto2001,
  author   = {Otto, Felix},
  title    = {{The geometry of dissipative evolution equations: the porous medium equation}},
  journal  = {{Communications in Partial Differential Equations}},
  volume   = {26},
  number   = {1--2},
  year     = {2001},
  pages    = {101--174},
  doi      = {10.1081/PDE-100002243},
  mrnumber = {1842429}
}

@article{Younes1998,
  author  = {Younes, Laurent},
  title   = {{Computable elastic distances between shapes}},
  journal = {{SIAM Journal on Applied Mathematics}},
  volume  = {58},
  number  = {2},
  year    = {1998},
  pages   = {565--586},
  doi     = {10.1137/S0036139995287685}
}

@article{Eells1966,
  author  = {Eells, Jr., James},
  title   = {A setting for global analysis},
  journal = {Bulletin of the American Mathematical Society},
  volume  = {72},
  number  = {5},
  pages   = {751--807},
  year    = {1966},
  month   = sep,
  doi     = {10.1090/S0002-9904-1966-11558-6}
}

@article{EellsElworthy1970,
  author   = {Eells, James and Elworthy, K. David},
  title    = {{Open embeddings of certain {Banach} manifolds}},
  journal  = {{Annals of Mathematics. Second Series}},
  volume   = {91},
  number   = {3},
  year     = {1970},
  pages    = {465--485},
  doi      = {10.2307/1970634},
  mrnumber = {0263120},
  zbl      = {0198.28804}
}

@book{Omori1974,
  author    = {Omori, Hideki},
  title     = {{Infinite Dimensional {Lie} Transformation Groups}},
  series    = {{Lecture Notes in Mathematics}},
  volume    = {427},
  publisher = {{Springer-Verlag}},
  address   = {{Berlin-New York}},
  year      = {1974},
  pages     = {xiv+154},
  doi       = {10.1007/BFb0063400},
  mrnumber  = {0431262},
  zbl       = {0328.58005}
}

@article{JerrardMaor2019,
  author  = {Jerrard, Robert L. and Maor, Cy},
  title   = {Vanishing geodesic distance for right-invariant {Sobolev} metrics on diffeomorphism groups},
  journal = {Annals of Global Analysis and Geometry},
  volume  = {55},
  number  = {4},
  pages   = {631--656},
  year    = {2019},
  doi     = {10.1007/s10455-018-9644-y}
}

@article{BauerBruverisHarmsMichor2012,
  author  = {Bauer, Martin and Bruveris, Martins and Harms, Philipp and Michor, Peter W.},
  title   = {Vanishing geodesic distance for the {Riemannian} metric with geodesic equation the {KdV}-equation},
  journal = {Annals of Global Analysis and Geometry},
  volume  = {41},
  number  = {4},
  pages   = {461--472},
  year    = {2012},
  doi     = {10.1007/s10455-011-9294-9}
}

@article{Eliasson1967,
  author   = {El{\'\i}asson, Halld{\'o}r I.},
  title    = {{Geometry of manifolds of maps}},
  journal  = {{Journal of Differential Geometry}},
  volume   = {1},
  year     = {1967},
  pages    = {169--194},
  doi      = {10.4310/jdg/1214427887},
  mrnumber = {0226681},
  zbl      = {0163.43901}
}

@book{Riemann2016,
  author    = {Riemann, Bernhard},
  title     = {On the Hypotheses Which Lie at the Bases of Geometry},
  editor    = {Jost, J{\"u}rgen},
  publisher = {Birkh{\"a}user},
  year      = {2016},
  note      = {Edited, historically and mathematically commented by J{\"u}rgen Jost}
}

@article{BegMillerTrouveYounes2005,
  author  = {Beg, M. Faisal and Miller, Michael I. and Trouv{\'e}, Alain and Younes, Laurent},
  title   = {{Computing large deformation metric mappings via geodesic flows of diffeomorphisms}},
  journal = {{International Journal of Computer Vision}},
  volume  = {61},
  number  = {2},
  year    = {2005},
  pages   = {139--157},
  doi     = {10.1023/B:VISI.0000043755.93987.aa}
}

@article{Michor_Mumford_vanishing_geodesic_distance,
  author  = {Michor, Peter W. and Mumford, David},
  title   = {Vanishing geodesic distance on spaces of submanifolds and diffeomorphisms},
  journal = {Documenta Mathematica},
  volume  = {10},
  pages   = {217--245},
  year    = {2005},
  doi     = {10.4171/DM/187}
}

@article{Smale1959,
  author  = {Smale, Stephen},
  title   = {Diffeomorphisms of the two-sphere},
  journal = {Proceedings of the American Mathematical Society},
  volume  = {10},
  year    = {1959},
  pages   = {621--626}
}

@article{Hatcher1983,
  author  = {Hatcher, Allen E.},
  title   = {A proof of the {Smale} conjecture, {$\operatorname{Diff}(S^3) \simeq O(4)$}},
  journal = {Annals of Mathematics},
  volume  = {117},
  number  = {3},
  pages   = {553--607},
  year    = {1983},
  doi     = {10.2307/2007035}
}

@article{Poincare1905,
  author   = {Poincar{\'e}, Henri},
  title    = {{Sur les lignes g{\'e}od{\'e}siques des surfaces convexes}},
  journal  = {{Transactions of the American Mathematical Society}},
  volume   = {6},
  number   = {3},
  year     = {1905},
  pages    = {237--274},
  doi      = {10.2307/1986219},
  mrnumber = {1500710},
  jfm      = {36.0669.01}
}

@book{AbrahamMarsden1978,
  author    = {Abraham, Ralph and Marsden, Jerrold E.},
  title     = {Foundations of Mechanics},
  edition   = {2nd},
  publisher = {Benjamin/Cummings Publishing Co., Inc.},
  address   = {Reading, MA},
  year      = {1978},
  note      = {Revised and enlarged, with the assistance of Tudor Ratiu and Richard Cushman}
}

@book{Arnold1989,
  author    = {Arnold, V. I.},
  title     = {Mathematical Methods of Classical Mechanics},
  series    = {Graduate Texts in Mathematics},
  volume    = {60},
  edition   = {2nd},
  publisher = {Springer-Verlag},
  address   = {New York},
  year      = {1989},
  note      = {Translated from the Russian by K. Vogtmann and A. Weinstein}
}

@book{McDuffSalamon2017,
  author    = {McDuff, Dusa and Salamon, Dietmar},
  title     = {Introduction to Symplectic Topology},
  series    = {Oxford Graduate Texts in Mathematics},
  volume    = {27},
  edition   = {3rd},
  publisher = {Oxford University Press},
  address   = {Oxford},
  year      = {2017}
}

@book{Rabinowitz1986,
  author    = {Rabinowitz, Paul H.},
  title     = {Minimax Methods in Critical Point Theory with Applications to Differential Equations},
  series    = {CBMS Regional Conference Series in Mathematics},
  volume    = {65},
  publisher = {American Mathematical Society},
  address   = {Providence, RI},
  year      = {1986}
}

@book{Struwe2008,
  author    = {Struwe, Michael},
  title     = {Variational Methods},
  series    = {Ergebnisse der Mathematik und ihrer Grenzgebiete. 3. Folge. A Series of Modern Surveys in Mathematics},
  volume    = {34},
  edition   = {4th},
  publisher = {Springer-Verlag},
  address   = {Berlin},
  year      = {2008},
  note      = {Applications to nonlinear partial differential equations and Hamiltonian systems}
}

@book{Sorrentino2015,
  author    = {Sorrentino, Alfonso},
  title     = {Action-Minimizing Methods in Hamiltonian Dynamics},
  series    = {Mathematical Notes},
  volume    = {50},
  publisher = {Princeton University Press},
  address   = {Princeton, NJ},
  year      = {2015},
  note      = {An introduction to Aubry--Mather theory}
}

@book{McDuffSalamon2012,
  author    = {McDuff, Dusa and Salamon, Dietmar},
  title     = {{J}-holomorphic Curves and Symplectic Topology},
  series    = {American Mathematical Society Colloquium Publications},
  volume    = {52},
  edition   = {2nd},
  publisher = {American Mathematical Society},
  address   = {Providence, RI},
  year      = {2012}
}

@book{AudinDamian2014,
  author    = {Audin, Mich{\`e}le and Damian, Mihai},
  title     = {Morse Theory and Floer Homology},
  series    = {Universitext},
  publisher = {Springer},
  address   = {London},
  year      = {2014},
  note      = {Translated from the French by Reinie Ern{\'e}}
}

@article{Hadamard1898,
  author  = {Hadamard, Jacques},
  title   = {{Les surfaces {\`a} courbures oppos{\'e}es et leurs lignes g{\'e}od{\'e}siques}},
  journal = {{Journal de Math{\'e}matiques Pures et Appliqu{\'e}es}},
  series  = {{5e s{\'e}rie}},
  volume  = {4},
  year    = {1898},
  pages   = {27--73},
  jfm     = {29.0522.01},
  url     = {https://www.numdam.org/item/JMPA_1898_5_4__27_0/}
}

@article{Birkhoff1917,
  author   = {Birkhoff, George D.},
  title    = {{Dynamical systems with two degrees of freedom}},
  journal  = {{Transactions of the American Mathematical Society}},
  volume   = {18},
  number   = {2},
  year     = {1917},
  pages    = {199--300},
  doi      = {10.1090/S0002-9947-1917-1501070-3},
  mrnumber = {1501070},
  jfm      = {46.1174.01}
}

@article{GromollMeyer1969,
  author   = {Gromoll, Detlef and Meyer, Wolfgang},
  title    = {{Periodic geodesics on compact {Riemannian} manifolds}},
  journal  = {{Journal of Differential Geometry}},
  volume   = {3},
  number   = {3--4},
  year     = {1969},
  pages    = {493--510},
  doi      = {10.4310/jdg/1214429070},
  mrnumber = {0264551},
  zbl      = {0203.54401}
}

@book{Klingenberg1978,
  author    = {Klingenberg, Wilhelm},
  title     = {{Lectures on Closed Geodesics}},
  series    = {{Grundlehren der Mathematischen Wissenschaften}},
  volume    = {230},
  publisher = {{Springer-Verlag}},
  address   = {{Berlin-New York}},
  year      = {1978},
  pages     = {ix+227},
  isbn      = {978-3-540-08393-1},
  doi       = {10.1007/978-3-642-61881-9},
  mrnumber  = {0478069},
  zbl       = {0397.58018}
}

@book{younes2010shapes,
  title={Shapes and diffeomorphisms},
  author={Younes, Laurent},
  volume={171},
  year={2010},
  publisher={Springer}
}

@article{micheli2013sobolev,
  author        = {Micheli, Mario and Michor, Peter W. and Mumford, David},
  title         = {{Sobolev} metrics on diffeomorphism groups and the derived geometry of spaces of submanifolds},
  journal       = {Izvestiya: Mathematics},
  volume        = {77},
  number        = {3},
  pages         = {541--570},
  year          = {2013},
  doi           = {10.1070/IM2013v077n03ABEH002648},
  archivePrefix = {arXiv},
  eprint        = {1202.3677}
}

@Article{HopfrinowfalseEkeland,
    AUTHOR = {Ekeland, Ivar},
     TITLE = {The {H}opf-{R}inow theorem in infinite dimension},
   JOURNAL = {J. Differential Geometry},
  FJOURNAL = {Journal of Differential Geometry},
    VOLUME = {13},
      YEAR = {1978},
    NUMBER = {2},
     PAGES = {287--301},
      ISSN = {0022-040X,1945-743X},
   MRCLASS = {58B20 (58E10)},
  MRNUMBER = {540948},
MRREVIEWER = {Karsten\ Grove},
       URL = {http://projecteuclid.org/euclid.jdg/1214434494},
}

@article{HopfrinowfalseAktkin,
  author  = {Atkin, C. J.},
  title   = {The {Hopf--Rinow} theorem is false in infinite dimensions},
  journal = {Bulletin of the London Mathematical Society},
  volume  = {7},
  number  = {3},
  pages   = {261--266},
  year    = {1975},
  month   = nov,
  doi     = {10.1112/blms/7.3.261}
}
\end{document}